\documentclass[10pt]{amsart}
\usepackage[margin=1.5in]{geometry}
\usepackage{xcolor}
\usepackage[notcite,notref]{showkeys}
\usepackage{xcolor}
\newtheorem{theorem}{Theorem}[section]
\newtheorem{lemma}[theorem]{Lemma}
\newtheorem{proposition}{Proposition}[section]
\newtheorem{corollary}{Corollary}[section]
\theoremstyle{definition}
\newtheorem{definition}[theorem]{Definition}

\theoremstyle{remark}
\newtheorem{remark}[theorem]{Remark}

\numberwithin{equation}{section}

\begin{document}
\title[Quantum BGK model near a global Fermi-Dirac distribution]{Quantum BGK model near a global Fermi-Dirac distribution}

\author{Gi-Chan Bae}
\address{Department of mathematics, Sungkyunkwan University, Suwon 440-746, Republic of Korea }
\email{gcbae02@skku.edu}

\author{Seok-Bae Yun}
\address{Department of mathematics, Sungkyunkwan University, Suwon 440-746, Republic of Korea }
\email{sbyun01@skku.edu}



\keywords{Quantum BGK model, Quantum Boltzmann equation, Fermi-Dirac distribution, Nonlinear energy method}

\begin{abstract}
In this paper, we consider the existence and asymptotic behavior of the fermionic quantum BGK
model, which is a relaxation model of the quantum Boltzmann equation for fermions. More precisely,
we establish the existence of unique classical solutions and their exponentially fast stabilization  when  the initial data starts sufficiently close to a global Fermi-Dirac distribution.
A key difficulty unobserved in the study of the classical BGK model is that we must verify that the equilibrium parameters is uniquely determined through a set of nonlinear equations in each iteration step. 
\end{abstract}

\maketitle
\section{Introduction}
\subsection{Quantum BGK model} 
The quantum modification of the celebrated Boltzmann equation was first suggested in \cite{FN,KN,U1,U2}, which often goes by the name of Uehling-Uhlenbeck equation or Nordheim equation.
But the intricate structure of the collision operator complicates the computations and understanding of quantum transport properties,
and the relaxation time approximation are widely used in physics and engineering 
\cite{Ashcroft,Chap,FJ,Ihn,Jungel-Quasi,Jungel-transport,Jin,Khal,Kittel,MO,MRS,N,RS,YMCL,YYDHLZZ}:
\begin{align}\label{QBGK}
\begin{split}
\partial_tF+p\cdot\nabla_x F&=\frac{1}{\tau}(\mathcal{F}(F)-F), \cr
F(x,p,0)&=F_0(x,p).
\end{split}
\end{align}
Here $F(x,p,t)$ is the number density function on phase point $(x,p)\in\mathbb{T}^3\times \mathbb{R}^3$ at time $t\in\mathbb{R}_+$.
$\tau$ is the relaxation time. The Fermi-Dirac distribution $\mathcal{F}(F)$, which is the quantum counterpart of the classical Maxwellian for fermions is defined by the
following process: First, we define the macroscopic fields of local density, momentum and  energy:
\begin{align}\label{NPE}
\begin{split}
N(x,t)&=\int_{\mathbb{R}^3}F(x,p,t)dp,\cr
P(x,t)&=\int_{\mathbb{R}^3}F(x,p,t)pdp,\cr
E(x,t)&=\int_{\mathbb{R}^3}F(x,p,t)|p|^2dp.
\end{split}
\end{align}
We then define the equilibrium constants: First, we derive $c(x,t)$ from the following nonlinear functional equation:
\begin{align}\label{a,c1}
\begin{split}
\frac{N(x,t)}{\left(E(x,t)-\frac{P^2(x,t)}{N(x,t)}\right)^{\frac{3}{5}}}= \frac{\displaystyle\int_{\mathbb{R}^3}\frac{1}{e^{|p|^2+c(x,t)}+1}dp}{\displaystyle\left(\int_{\mathbb{R}^3}\frac{|p|^2}{e^{|p|^2+c(x,t)}+1}dp\right)^{\frac{3}{5}}}.
\end{split}
\end{align}
In view of this relation, we define $\beta(c)$ and $B(N,P,E)$ for later convenience as
\begin{align}\label{beta}
\beta(c)= \frac{\int_{\mathbb{R}^3}\frac{1}{e^{|p|^2+c}+1}dp}{\left(\int_{\mathbb{R}^3}\frac{|p|^2}{e^{|p|^2+c}+1}dp\right)^{\frac{3}{5}}},\quad
B(N,P,E)=\frac{N(x,t)}{\left(E(x,t)-\frac{P(x,t)^2}{N(x,t)}\right)^{\frac{3}{5}}}.
\end{align}
Once  $c(x,t)$ is determined by the relation (\ref{a,c1}), we define $a(x,t)$ by
\begin{align}\label{a,c2}
\displaystyle a(x,t)&=\left({\int_{\mathbb{R}^3}\frac{1}{e^{|p|^2+c(x,t)}+1}dp}\right)^\frac{2}{3}N(x,t)^{-\frac{2}{3}} .
\end{align}
It will be shown later that (\ref{a,c1}) and (\ref{a,c2}) uniquely determines $c$ under additional conditions (See Theorem \ref{unique c}).

The Fermi-Dirac distribution is now defined as follows:
\begin{align}\label{FD local}
\displaystyle\mathcal{F}(F)(x,p,t)=\frac{1}{e^{a(x,t)\big|p-\frac{P(x,t)}{N(x,t)}\big|^2+c(x,t)}+1}.
\end{align}
The relaxation opeartor of the quantum-BGK model satisfies the following cancellation property (See Section 2).
\begin{align}\label{cancelation}
\int_{\mathbb{R}^3}\mathcal{F}(F)(x,p,t)\left(\begin{array}{c}1\cr p\cr |p|^2\end{array}\right)dp
=\int_{\mathbb{R}^3}F(x,p,t)\left(\begin{array}{c}1\cr p\cr |p|^2\end{array}\right)dp,
\end{align}
which implies the conservation laws of $N(x,t)$, $P(x,t)$ and $E(x,t)$:
\begin{align}
\begin{split}\label{Conservation}
\int_{\mathbb{T}^3\times \mathbb{R}^3}F(t) dxdp &= \int_{\mathbb{T}^3\times \mathbb{R}^3}F_0 dxdp, \cr
\int_{\mathbb{T}^3\times \mathbb{R}^3}F(t)p dxdp &= \int_{\mathbb{T}^3\times \mathbb{R}^3}F_0p dxdp, \cr
\int_{\mathbb{T}^3\times \mathbb{R}^3}F(t)|p|^2 dxdp &= \int_{\mathbb{T}^3\times \mathbb{R}^3}F_0|p|^2 dxdp .
\end{split}
\end{align}
The following celebrated $H$-theorem was established in \cite{WMZ} :
\begin{align*}
\frac{d}{dt}H(F(t)) \leq 0,
\end{align*}
where the $H$-functional is defined by
\begin{align*}
H(F)=\int_{\mathbb{T}^3\times\mathbb{R}^3}F\ln F+(1-F)\ln (1-F) dxdp.
\end{align*}
We note that the $H$-functional is minimized when $F$ is a Fermi-Dirac distribution (See Sec. 2).\newline

The relaxation time $\tau$ can take various different forms depending on the physical situations, but usually given as an energy dependent, and hence, temperature dependent function.  Through out this paper, we assume that the relaxation time takes the following form:
\begin{align}\label{relaxation time0}
\frac{1}{\tau} = P(N)(C_1T^n+C_2T^m+C_3)+C_4,
\end{align}
where $T$ denotes the local temperature, and $P$ is a homogeneous generic polynomial and $m,n, C_i~(i=1,2,3,4)$ satisfies
\begin{align*}
n\geq0,~m\leq0,~C_i\geq0, ~\sum C_i\neq 0.
\end{align*}
Since the temperature and the equilibrium coefficients given in (\ref{a,c2}) are related by $T=(k_B a)^{-1}$ through the Boltzmann constant $k_B$ \cite{Jin}, we rewrite (\ref{relaxation time0}) as
\begin{align}\label{relaxation time}
\frac{1}{\tau} = P(N)(C_1a^n+C_2a^m+C_3)+C_4.
\end{align}
This encompass a wide range of the expressions for the relaxation time in the literature \cite{BCCHOY,Brauk1,HF,Iti,Jin,MK,N,RMR,Hubb,Spa,YMCL,YYDHLZZ}.
\subsection{Novelty and difficulty} The goal of this paper is to establish the existence of classical solutions and their asymptotic behavior using the nonlinear energy method \cite{Guo whole,Guo VMB,Guo VPB}, when the initial data lies close to a global Fermi-Dirac distribution:
\begin{align}\label{global FD}
m(p)&=\frac{1}{e^{a_0|p|^2+c_0}+1},
\end{align}
where $a_0$ and $c_0$ are determined by the following relation:
\begin{align}\label{by}
\frac{N_0}{\left(E_0-\frac{P_0^2}{N_0}\right)^{\frac{3}{5}}}=\frac{\int_{\mathbb{R}^3}\frac{1}{e^{|p|^2+c_0}+1}dp}{\left(\int_{\mathbb{R}^3}\frac{|p|^2}{e^{|p|^2+c_0}+1}dp\right)^\frac{3}{5}},	\quad
a_0=\left({\int_{\mathbb{R}^3}\frac{1}{e^{|p|^2+c_0}+1}dp}\right)^\frac{2}{3}N_0^{-\frac{2}{3}}.
\end{align}
Here  $N_0$, $P_0$ and $E_0$ are defined as in (\ref{NPE}) from the initial data:
\begin{align}\label{NPE0}
N_0=\int_{\mathbb{T}^3\times \mathbb{R}^3}F_0 dxdp,\quad P_0=\int_{\mathbb{T}^3\times \mathbb{R}^3}F_0p dxdp, \quad E_0=\int_{\mathbb{T}^3\times \mathbb{R}^3}F_0|p|^2 dxdp.
\end{align}
Note that $P_0=0$.\newline

For this, we decompose $F$ into the equilibrium and  the perturbation as
\begin{align}\label{novel decomposition}
F=m+\sqrt{m-m^2}f,
\end{align}
and write (\ref{QBGK}) as follows:
\begin{align*}
\partial_tf+p\cdot\nabla_x f&=Lf+\Gamma(f), \cr
f(x,p,0)&=f_0(x,p),
\end{align*}
where $L$ denotes the linearized relaxation operator:
\begin{align*}
Lf= Pf-f,
\end{align*}
and $\Gamma(f)$ is nonlinear term. (Precise definitions is in section 2.)
$P$ is macroscopic projection operator for $f$ on the five-dimensional linear space spanned by
\begin{align}\label{null}
\big\{\sqrt{m-m^2},p_1\sqrt{m-m^2},~p_2\sqrt{m-m^2},p_3\sqrt{m-m^2},|p|^2\sqrt{m-m^2}~\big\}.
\end{align}


We take $\sqrt{m-m^2}$ as the weight function in the perturbation instead of usual $\sqrt{m}$, to treat the nonlinear structure  $\mathcal{F}-\mathcal{F}^2$ arising from the differentiation of the local Fermi-Dirac distribution with respect to the macroscopic fields: $N$, $P$ and $E$. Such nonlinear structure turns out to be inconsistent with the conventional weight $\sqrt{m}$,
and the choice of weight function $\sqrt{m-m^2}$ enables one to resolve such inconsistence, leading to the desired dissipative structure of $L$. Similar observation was made in \cite{Lemou,Liu} for quantum Landau equations (See Section 3).

On the other hand, we see that the equation (\ref{QBGK}) is well-defined only when we are able to find the equilibrium coefficients $a$ and $c$ uniquely from (\ref{cancelation}).
In view of this, we must guarantee that the functional relations (\ref{a,c1}) and (\ref{a,c2}) uniquely determine the equilibrium coefficients  in each iteration step.
We accomplish this by $1)$ proving in Proposition \ref{betathm} that the function $\beta(x)$ defined in (\ref{beta}) is strictly decreasing if we restrict $x$ to $(-\ln3,\infty)$:
\[
\beta^{\prime}(x)<0\quad \mbox{for }x>-\ln3,
\]
and $2)$ showing that the l.h.s of (\ref{a,c1}) falls into the range of $\beta$ for each $n$:
\begin{align*}
0<B(N_n,P_n,E_n) < \beta(-\ln3),
\end{align*}
if such inequality is satisfied initially, and the high-order energy are kept sufficiently small for each iteration.
This enables us to find a unique $c_n$ in the region $(-\ln3,\infty)$ and, in turn, $a_n$ so that we can proceed to the next iteration step (See Section 5.2).
\subsection{Main results}
We first need to set up some notational conventions.
\begin{itemize}
\item  The constants in the estimates will be defined generically.
\item $\langle \cdot,\cdot\rangle_{L^2_{p}}$ and $\langle\cdot,\cdot\rangle_{L^2_{x,p}}$ denote the standard $L^2$ inner product on $\mathbb{R}^3_p$ and  $\mathbb{T}^3_x \times \mathbb{R}^3_p$ respectively.
\begin{align*}
\langle f,g\rangle_{L^2_{p}}=\int_{\mathbb{R}^3}f(p)g(p)dp,	\quad\langle f,g\rangle_{L^2_{x,p}}=\int_{\mathbb{T}^3\times\mathbb{R}^3}f(x,p)g(x,p)dxdp.
\end{align*}
\item $||\cdot||_{L^2_p}$ and $||\cdot||_{L^2_{x,p}}$ denote the standard $L^2$ norms in $\mathbb{R}^3_p$ and $\mathbb{T}^3_x \times \mathbb{R}^3_p$ respectively:
\begin{align*}
||f||_{L^2_p}\equiv \left(\int_{\mathbb{R}^3}|f(p)|^2 dp\right)^{\frac{1}{2}},	\quad||f||_{L^2_{x,p}}\equiv \left(\int_{\mathbb{T}^3\times\mathbb{R}^3}|f(x,p)|^2 dxdp\right)^{\frac{1}{2}}.
\end{align*}
\item We use the following notations for multi-indices, differential operators:
\begin{align*}
\alpha=[\alpha_0,\alpha_1,\alpha_2,\alpha_3], \quad \beta=[\beta_1,\beta_2,\beta_3],
\end{align*}
and
\begin{align*}
\partial^{\alpha}_{\beta}=\partial_t^{\alpha_0}\partial_{x_1}^{\alpha_1}\partial_{x_2}^{\alpha_2}\partial_{x_3}^{\alpha_3}\partial_{p_1}^{\beta_1}\partial_{p_2}^{\beta_2}\partial_{p_3}^{\beta_3}.
\end{align*}
\end{itemize}

We define the high-order energy functional $\mathcal{E}(f(t))$(or $\mathcal{E}(t)$):

\begin{align*}
\mathcal{E}(f(t))=\sum_{|\alpha|+|\beta|\leq N}||\partial^{\alpha}_{\beta}f(t)||^2_{L^2_{x,p}}.
\end{align*}
We are ready to state our main result.
\begin{theorem}\label{mainthm}
Let $N\geq 3$. Suppose that $F_0= m+\sqrt{m-m^2}f_0 \geq 0$ satisfies
\begin{align}\label{suppini}
\frac{N_0}{E_0^{3/5}}<\beta(-\ln3).
\end{align}
Then there exists positive constant $\delta$ and $C$, such that if $\mathcal{E}(f_0)\leq\delta$, then there exists a unique global solution $F$ to (\ref{QBGK}) such that
\begin{enumerate}
\item The distribution function $F$ is non-negative for all $t>0$:
\begin{align*}
F=m+\sqrt{m-m^2}f \geq 0,
\end{align*}
and satisfies
\[
 0<B(N,P,E)<\beta(-\ln3).
\]
\item The conservation laws (\ref{Conservation}) hold.

\item The high order energy functional $\mathcal{E}(f(t))$ is uniformly bounded:
\begin{align*}
\sup_{t\in \mathbb{R}_+}\mathcal{E}(f(t))\leq C\mathcal{E}(f_0).
\end{align*}

\item The perturbation decays exponentially fast:
\begin{align*}
\sum_{|\alpha|+|\beta|\leq N}\|\partial^{\alpha}_{\beta}f(t)\|_{L^2_{x,p}} \leq Ce^{-\epsilon t},
\end{align*}

for some positive constants $C$ and $\epsilon$.

\end{enumerate}
\end{theorem}
\begin{remark}
It states that if $B(N_0,P_0,E_0)$ lies in $(0,\beta(-\ln3))$, then $B(N,P,E)$ also lies in $(0,\beta(-\ln3))$.
It is only under such restriction that we are able to conclude that the local Fermi-Dirac distribution is uniquely determined to satisfied the conservation laws (See Section 2).
\end{remark}
\subsection{Brief history}
The prototype of relaxation type models in quantum theory can be traced back to early 1900s when Drude successfully explained the
fundamental transport properties of electrons such as the Ohm's law or Hall effect using his relaxation model. Ever since, relaxational approximations has been a popular tool in quantum and condensed matter physics to understand various transport phenomena.
Despite such popularity of the quantum relaxation model in physics and engineering, the mathematical research on the model has a rather short history, and most of the important problems remain unanswered. We refer to \cite{N} for the study on a stationary problem for bosonic quantum BGK model with modified condensation ansantz. In \cite{Brauk1,Brauk2}, the author considers the existence and asymptotic behavior of analytic solutions for a BGK type model
arising in the study of the cloud of ultra-cold atoms in an optical lattice.
These results seem to be the all existence results known so far for quantum BGK models.
For numerical computations for quantum BGK models, we refer to \cite{CM,F2,F1,MY,Ringhofer,SFB,SY,WMZ,YH,Yano}.

Literature on quantum Boltzmann equations, especially in the case of free quantum particles, are much richer.
For studies in the spatially homogeneous regime, we refer to \cite{AG,BE,EMV,EV,LL,LuB1,LuB7,LuB3,LuB5,LuB6,LuB2,LuB4}  for bosonic gas, and  \cite{EMV,LuF1,LuFR1,LuFR2,LuW} for fermions.
Linearized problem for the spatially homogeneous quantum Boltzmann equation were investigated in \cite{ET,EMVela,EMVela2}.
In the case of spatially inhomogeneous case, the existence of mild solution and its long time behavior is obtained by Dolbeault in \cite{D}. Lions derived
the existence of renormalized solution in \cite{Lions}.  Allemand considered conservation laws and hydrodynamic limits in \cite{All2}.
The Quantum Boltzmann equation in spatially decaying regime was investigated in \cite{ZL1,ZL2} for existence and long time behavior and in \cite{Ha} for uniform $L^1$ stability estimate.
For the derivation of quantum Boltzmann equation, see \cite{BCEP,Ca,H}. Quantum hydrodynamic models limit considered in \cite{All1,DR,Za}.
Studies on Wigner-Poisson type equation can be found in \cite{Ar,BM,Ill,LTZ,LTZ2,LZ,M,Nt}.
We refer to  \cite{BGK,CZ,KP,Mischler,Perthame,PP,RSY,WZ,Yun1,Yun11,Yun2,Yun3,Yun4,Zhang,ZH} for mathematical results on classical BGK models.
Nice survey on classical or quantum kinetic equations can be found in \cite{Chap,C,CIP,GL,Sone,Sone2,Stru-book,U-T,V}.\newline





This paper is organized as follows: In Section 2, we study the  well-posedness problem for the Fermi-Dirac distribution. In section 3, relaxation operator is linearized around a global Fermi-Dirac distribution. In section 4, we present a priori estimates for macroscopic quantities and equilibrium coefficients. In section 5, local in time existence and uniqueness is derived. Finally, we prove our main theorem in Section 6.

%
%
%
%
%
%
\section{Monotonicity of $\beta$}
In this section, we consider the problem of determination of the equilibrium coefficients $a$ and $c$ in the local Fermi-Dirac distribution.
For this, we study the minimization problem of $H$-functional
\begin{align*}
H(F)=\int_{\mathbb{R}^3}(1-F)\ln(1-F)+F\ln F dp
\end{align*}
under the constraints of (\ref{cancelation}).
The corresponding Euler-Lagrange equation is
\begin{align*}
\ln{\frac{F}{1-F}} +\lambda_1+(\lambda_2,\lambda_3,\lambda_4)\cdot p + \lambda_5|p|^2 = 0,
\end{align*}
which can be rewritten as
\begin{align*}
\displaystyle F(p)=\frac{1}{e^{\lambda_1+(\lambda_2,\lambda_3,\lambda_4)\cdot p + \lambda_5|p|^2} +1}.
\end{align*}
It remains to choose $\lambda_i$ $(i=1,...,5)$  so that $\mathcal{F}$ shares the zeroth, first, and second moments with $F$ as in (\ref{cancelation}).
For simplicity, we reparametrize $\lambda_1,...,\lambda_5$  to write the Fermi-Dirac distribution as follows:
\begin{align*}
\mathcal{F}(p)&=\frac{1}{e^{a|p-b|^2+c}+1},
\end{align*}
for $a\in\mathbb{R}^+$ , $b\in\mathbb{R}^3$, $c\in\mathbb{R}$.
We now check whether  $a$, $b$, $c$ can be uniquely determined by $N$, $P$, $E$. First, by making a change of variable $\sqrt{a}(p-b)\rightarrow p$, we get from the first line of (\ref{cancelation}) that
\begin{align}\label{N}
N(x,t)=\int_{\mathbb{R}^3}\frac{1}{e^{a|p-b|^2+c}+1}dp	
=a^{-\frac{3}{2}}\int_{\mathbb{R}^3}\frac{1}{e^{|p|^2+c}+1}dp.
\end{align}
Similarly, we make change of variable $p-b\rightarrow p$ and use the oddness of $p/(e^{a|p|^2+c}+1)$ to write the second line of (\ref{cancelation}) as
\begin{align*}
P(x,t)=\int_{\mathbb{R}^3}\frac{p}{e^{a|p-b|^2+c}+1}dp	
=\int_{\mathbb{R}^3}\frac{p+b}{e^{a|p|^2+c}+1}dp
= b N(x,t).
\end{align*}
which gives the representation of $b$:
\begin{align}\label{b}
b(x,t)=\frac{P(x,t)}{N(x,t)}.
\end{align}
Finally, we compute the last line of (\ref{cancelation}) as follows:
\begin{align*}
E(x,t)&=\int_{\mathbb{R}^3}\frac{|p|^2}{e^{a|p-b|^2+c}+1}dp	\cr
&= \int_{\mathbb{R}^3}\frac{|p+b|^2}{e^{a|p|^2+c}+1}dp	\cr
&= \int_{\mathbb{R}^3}\frac{|p|^2}{e^{a|p|^2+c}+1}dp+\int_{\mathbb{R}^3}\frac{2p\cdot b}{e^{a|p|^2+c}+1}dp+\int_{\mathbb{R}^3}\frac{b^2}{e^{a|p|^2+c}+1}dp	\cr
&= a^{-\frac{5}{2}}\int_{\mathbb{R}^3}\frac{|p|^2}{e^{|p|^2+c}+1}dp+Nb^2(x,t),
\end{align*}
which, combined with (\ref{b}), gives
\begin{align}\label{this}
E(x,t)-\frac{P(x,t)^2}{N(x,t)}&=a^{-\frac{5}{2}}\int_{\mathbb{R}^3}\frac{|p|^2}{e^{|p|^2+c}+1}dp.
\end{align}
From  (\ref{N}) and (\ref{this}), we deduce that
\begin{align*}
\frac{N(x,t)}{\left(E(x,t)-\frac{P(x,t)^2}{N(x,t)}\right)^{\frac{3}{5}}}= \frac{\int_{\mathbb{R}^3}\frac{1}{e^{|p|^2+c}+1}dp}{{\left(\int_{\mathbb{R}^3}\frac{|p|^2}{e^{|p|^2+c}+1}dp\right)^\frac{3}{5}}},
\end{align*}
or, in view of (\ref{beta})
\begin{align}\label{remain to check}
\beta(c)=B(N,P,E).
\end{align}
If we can determine $c$ from this identity, we can recover $a$ from (\ref{N}) by
\begin{align*}
a(x,t)=\left(\int_{\mathbb{R}^3}\frac{1}{e^{|p|^2+c(x,t)}+1}dp\right)^{\frac{2}{3}}N(x,t)^{-\frac{2}{3}}.
\end{align*}
Therefore, it remains to check that ($\ref{remain to check}$) uniquely determines $c$, which is accomplished in the following theorem.
\begin{theorem}\label{unique c}
Assume $0<B(N,P,E)<\beta(-\ln3)$. Then,
\[
\beta(c)=B(N,P,E)
\]
has a unique solution $c$ in $(-\ln3,\infty)$.
\end{theorem}
\begin{proof}
This follows directly from the fact that
\[
\lim_{c\rightarrow\infty}\beta(c)=0
\]
and Proposition \ref{betathm} below.
\end{proof}
In view of this theorem, we allow a slight abuse the notation to use $\beta^{-1}$ in the following sense:
\begin{align}\label{beta inverse}
\beta^{-1}=\big(\beta\big|_{(-\ln3,\infty)}\big)^{-1}.
\end{align}
\begin{proposition}\label{betathm}
The function $\beta(c)$ defined in (\ref{beta}) is a strictly decreasing function of $c$ when $c\geq -\ln3$.
\end{proposition}
\begin{remark}
The monotonicity of $\beta$ in the case $c<-\ln3$ is inconclusive for now.
\end{remark}
\begin{proof}
We will show that the $\beta^{\prime}(c)$ is strictly negative in $c\geq-\ln3$. The infinite differentiability of $\beta$ with respect to $c$ is
clear from the definition of $\beta$.
By an explicit computation we see that
\begin{align}\label{beta prime}
\begin{split}
\beta'(c) =\frac{\left(\int_{\mathbb{R}^3}\frac{|p|^2}{e^{|p|^2+c}+1}dp\right)\left(\int_{\mathbb{R}^3}\frac{-e^{|p|^2+c}}{(e^{|p|^2+c}+1)^2}dp\right)	-\frac{3}{5}\left(\int_{\mathbb{R}^3}\frac{-|p|^2e^{|p|^2+c}}{(e^{|p|^2+c}+1)^2}dp\right)
\left(\int_{\mathbb{R}^3}\frac{1}{e^{|p|^2+c}+1}dp\right)}{\left(\int_{\mathbb{R}^3}\frac{|p|^2}{e^{|p|^2+c}+1}dp\right)^{\frac{8}{5}}}.
\end{split}
\end{align}
We represent in the spherical coordinates:
\begin{align*}
\beta'(c)=\frac{(4\pi)^{2/5}D(c)}{2\Big(\int_0^{\infty}\frac{r^4}{e^{r^2+c}+1}dr\Big)^\frac{8}{5}},
\end{align*}
where
\begin{align*}
D(c)=\frac{6}{5}\int_0^{\infty}\frac{r^4e^{r^2+c}}{(e^{r^2+c}+1)^2}dr\int_0^{\infty}\frac{r^2}{e^{r^2+c}+1}dr
-2\int_0^{\infty}\frac{r^4}{e^{r^2+c}+1}dr\int_0^{\infty}\frac{r^2e^{r^2+c}}{(e^{r^2+c}+1)^2}dr.
\end{align*}
Therefore, the desired result is achieved if we show that $D(c)<0$.
We then apply the integration by parts: $u'=\frac{2re^{r^2+c}}{(e^{r^2+c}+1)^2}$, $v=\frac{1}{2}r^3$ for
\begin{align*}
\int_0^{\infty}\frac{r^4e^{r^2+c}}{(e^{r^2+c}+1)^2}dr=\frac{3}{2}\int_0^{\infty}\frac{r^2}{e^{r^2+c}+1}dr,
\end{align*}
and $u'=\frac{2re^{r^2+c}}{(e^{r^2+c}+1)^2}$, $v=\frac{1}{2}r$ for
\begin{align*}
\int_0^{\infty}\frac{r^2e^{r^2+c}}{(e^{r^2+c}+1)^2}dr=\frac{1}{2}\int_0^{\infty}\frac{1}{e^{r^2+c}+1}dr,
\end{align*}
to rewrite $D(c)$ as
\[
D(c)=\frac{9}{5}\left(\int_0^{\infty}\frac{r^2}{e^{r^2+c}+1}dr\right)^2-\int_0^{\infty}\frac{r^4}{e^{r^2+c}+1}dr\int_0^{\infty}\frac{1}{e^{r^2+c}+1}dr.
\]
We then symmetrize $D(c)$:
\begin{align*}
D(c)&=\frac{9}{5}\int_0^{\infty}\frac{x^2}{e^{x^2+c}+1}dx\int_0^{\infty}\frac{y^2}{e^{y^2+c}+1}dy
-\int_0^{\infty}\frac{x^4}{e^{x^2+c}+1}dx\int_0^{\infty}\frac{1}{e^{y^2+c}+1}dy\cr
&=\int_0^{\infty}\int_0^{\infty}\frac{\frac{9}{5}x^2y^2-x^4}{(e^{x^2+c}+1)(e^{y^2+c}+1)}dxdy,
\end{align*}
and write in the spherical coordinate:
\begin{align*}
D(c)&=\int_0^{\frac{\pi}{2}}\int_0^{\infty}r^5\frac{\frac{9}{5}\cos^2\theta\sin^2\theta-\cos^4\theta}{(e^{r^2\cos^2\theta+c}+1)(e^{r^2\sin^2\theta+c}+1)}drd\theta.
\end{align*}
Applying the change of variable $\frac{\pi}{2}-\theta = t$, we get
\begin{align*}
D(c)= \int_0^{\frac{\pi}{2}}\int_0^{\infty}r^5\frac{\frac{9}{5}\sin^2t\cos^2t-\sin^4t}{(e^{r^2\sin^2t+c}+1)(e^{r^2\cos^2t+c}+1)}drdt.
\end{align*}
From these two expression, we obtain the following symmetric expression of $D$:
\begin{align}\label{betamono} D(c)=\frac{1}{2}\int_0^{\frac{\pi}{2}}\int_0^{\infty}r^5\frac{\frac{18}{5}\cos^2\theta\sin^2\theta-\cos^4\theta-\sin^4\theta}{(e^{r^2\cos^2\theta+c}+1)(e^{r^2\sin^2\theta+c}+1)}drd\theta.
\end{align}
We observe that
\begin{align*}
\frac{18}{5}\sin^2\theta \cos^2\theta-\sin^4\theta-\cos^4\theta
	&=-\frac{1}{10}(3+7\cos4\theta),
\end{align*}
to simplify this further into
\begin{align*}
D(c)=\frac{1}{2}\int_0^{\frac{\pi}{2}}\int_0^{\infty}r^5\frac{-\frac{1}{10}(3+7\cos4\theta)}{(e^{r^2\cos^2\theta+c}+1)(e^{r^2\sin^2\theta+c}+1)}drd\theta.
\end{align*}
Next, from the observation that
\begin{align*} &\int_0^{\frac{\pi}{4}}\int_0^{\infty}r^5\frac{-\frac{1}{10}(3+7\cos4\theta)}{(e^{r^2\cos^2\theta+c}+1)(e^{r^2\sin^2\theta+c}+1)}drd\theta\cr
&\qquad=\int_{\frac{\pi}{4}}^{\frac{\pi}{2}}\int_0^{\infty}r^5\frac{-\frac{1}{10}(3+7\cos4\theta)}{(e^{r^2\cos^2\theta+c}+1)(e^{r^2\sin^2\theta+c}+1)}drd\theta,
\end{align*}
which can be checked by considering the change of variable $\theta\rightarrow \pi/2-\theta$, we restrict the domain of integral of $D(c)$ into $[0,\pi/4]$:
\begin{align}\label{betamono}
D(c)=\int_0^{\frac{\pi}{4}}\int_0^{\infty}r^5\frac{-\frac{1}{10}(3+7\cos4\theta)}{(e^{r^2\cos^2\theta+c}+1)(e^{r^2\sin^2\theta+c}+1)}drd\theta.
\end{align}

In view of the fact that $-1/10(3+7\cos4\theta)$ changes sign from negative to positive at $\theta$$=\frac{1}{4}\cos^{-1}(-\frac{3}{7})$ in interval $[0,\frac{\pi}{4}]$, we divide the integral as into the negative part and the positive part:
\begin{align*}
I&=\int_0^{\frac{1}{4}\cos^{-1}(-\frac{3}{7})}\int_0^{\infty}r^5\frac{-\frac{1}{10}(3+7\cos4\theta)}{(e^{r^2\cos^2\theta+c}+1)(e^{r^2\sin^2\theta+c}+1)}drd\theta,	\cr
II&=\int_{\frac{1}{4}\cos^{-1}(-\frac{3}{7})}^{\frac{\pi}{4}}\int_0^{\infty}r^5\frac{-\frac{1}{10}(3+7\cos4\theta)}{(e^{r^2\cos^2\theta+c}+1)(e^{r^2\sin^2\theta+c}+1)}drd\theta.
\end{align*}
First, we observe that
\[
e^{r^2\cos^2\theta+c}+e^{r^2\sin^2\theta+c}\leq e^{r^2+c}+e^{c} \mbox{ on }~0\leq \theta\leq \frac{1}{4}\cos^{-1}\bigg(-\frac{3}{7}\bigg),
\]
to estimate the negative part $I$:
\begin{align*}
I&=\int_0^{\infty}\int_0^{\frac{1}{4}\cos^{-1}(-\frac{3}{7})}\frac{r^5(-\frac{1}{10}(3+7\cos4\theta))}{(e^{r^2+2c}+e^{r^2\cos^2\theta+c}+e^{r^2\sin^2\theta+c}+1)}d\theta dr \cr
&\leq \int_0^{\infty}\int_0^{\frac{1}{4}\cos^{-1}(-\frac{3}{7})}\frac{r^5(-\frac{1}{10}(3+7\cos4\theta))}{e^{r^2+2c}+e^{r^2+c}+e^{c}+1}d\theta dr \cr
& \leq \left(\int_0^{\infty}\frac{r^5}{e^{r^2+2c}+e^{r^2+c}+e^{c}+1}dr\right)\left(\int_0^{\frac{1}{4}\cos^{-1}(-\frac{3}{7})}-\frac{1}{10}(3+7\cos4\theta)d\theta\right).
\end{align*}
Similarly, we use
\[
e^{r^2\cos^2\theta+c}+e^{r^2\sin^2\theta+c} \geq 2\sqrt{e^{r^2\cos^2\theta+c}e^{r^2\sin^2\theta+c}}= 2e^{\frac{r^2}{2}+c},
\]
to estimate $II$ as
\begin{align*}
II&=\int_0^{\infty}\int_{\frac{1}{4}\cos^{-1}(-\frac{3}{7})}^{\frac{\pi}{4}}\frac{r^5(-\frac{1}{10}(3+7\cos4\theta))}{(e^{r^2+2c}+e^{r^2\cos^2\theta+c}+e^{r^2\sin^2\theta+c}+1)}d\theta dr	\cr
&\leq \int_0^{\infty}\int_{\frac{1}{4}\cos^{-1}(-\frac{3}{7})}^{\frac{\pi}{4}}\frac{r^5(-\frac{1}{10}(3+7\cos4\theta))}{e^{r^2+2c}+2e^{\frac{r^2}{2}+c}+1}d\theta dr	\cr	
&\leq \left(\int_0^{\infty}\frac{r^5}{e^{r^2+2c}+2e^{\frac{r^2}{2}+c}+1}dr\right)\left(\int_{\frac{1}{4}\cos^{-1}(-\frac{3}{7})}^{\frac{\pi}{4}} -\frac{1}{10}(3+7\cos4\theta)d\theta \right).
\end{align*}
Now, for simplicity, we set
\begin{align*}
\alpha&=\int_0^{\frac{1}{4}\cos^{-1}(-\frac{3}{7})}-\frac{1}{10}(3+7\cos4\theta) d\theta
=-\frac{3}{40}\cos^{-1}\left(-\frac{3}{7}\right)	  -\frac{1}{\sqrt{40}}<0,\cr	\beta&=\int_{\frac{1}{4}\cos^{-1}(-\frac{3}{7})}^{\frac{\pi}{4}}-\frac{1}{10}(3+7\cos4\theta) d\theta
=\frac{3\pi}{40}+\frac{3}{40}\cos^{-1}\left(-\frac{3}{7}\right)	  +\frac{1}{\sqrt{40}}
>0.
\end{align*}
We combine the above two estimates for $I$ and $II$ and observe $-4/\alpha>\beta>0$ to get
\begin{align*}
D(c)&=I+II\cr
&< \left(\int_0^{\infty}\frac{r^5}{e^{r^2+2c}+e^{r^2+c}+e^c+1}dr\right)\alpha+
\left(\int_0^{\infty}\frac{r^5}{e^{r^2+2c}+2e^{\frac{r^2}{2}+c}+1}dr\right)\beta\cr
&\leq\left(\int_0^{\infty}\frac{-4r^5}{e^{r^2+2c}+e^{r^2+c}+e^c+1}dr+
\int_0^{\infty}\frac{r^5}{e^{r^2+2c}+2e^{\frac{r^2}{2}+c}+1}dr\right)\left(-\frac{\alpha}{4}\right)\cr
&=  \left(\int_0^{\infty}r^5\frac{-4e^{r^2+2c}-8e^{\frac{r^2}{2}+c}-4+e^{r^2+2c}+e^{r^2+c}+e^c+1}{(e^{r^2+2c}+e^{r^2+c}+e^c+1)(e^{r^2+2c}+2e^{\frac{r^2}{2}+c}+1)}dr\right)
\left(-\frac{\alpha}{4}\right)	\cr
&\leq  \left(\int_0^{\infty}r^5\frac{-3e^{r^2+2c}+e^{r^2+c}+e^c-3}{(e^{r^2+2c}+e^{r^2+c}+e^c+1)(e^{r^2+2c}+2e^{\frac{r^2}{2}+c}+1)}dr\right)
\left(-\frac{\alpha}{4}\right).	
\end{align*}
In second line, strict inequality arise because $I$ and $II$ can not satisfy equality at the same time.
Recalling (\ref{beta prime}), what we have derived so far amounts to
\begin{align*}
\beta'(c)& < \frac{(4\pi)^{2/5}}{2}\Big(\int_0^{\infty}\frac{r^4}{e^{r^2+c}+1}dr\Big)^{-\frac{8}{5}}\left(-\frac{\alpha}{4}\right)\cr
&\quad\times\left(\int_0^{\infty}r^5\frac{-3e^{r^2+2c}+e^{r^2+c}+e^c-3}{(e^{r^2+2c}+e^{r^2+c}+e^c+1)(e^{r^2+2c}+2e^{\frac{r^2}{2}+c}+1)}dr\right).
\end{align*}
Therefore, we get the desired result from the following claim:\newline

{\bf Claim:} If $c\geq-\ln 3$, then $-3e^{r^2+2c}+e^{r^2+c}+e^c-3\leq 0$ for all $r\geq 0$.\newline
To prove this claim, we set
\begin{align*}
Y&=-3e^{r^2+2c}+e^{r^2+c}+e^c-3.
\end{align*}
Define
\[
X=e^{r^2},
\]
to rewrite $Y$ as
\[
Y=\big(-3e^{2c}+e^c\big)X+e^c-3 \quad (X\geq 1).
\]
For this straight line to stay strictly negative for all $X\geq 1$, we impose the following condition
\[
-3e^{2c}+e^c\leq 0,~\mbox{ and }~Y(1)=-3e^{2c}+2e^c-3<0.
\]
Since the second inequality is automatically satisfied, we only need to consider the first one,
which is equivalent to $c\geq-\ln 3$. This completes the proof of the claim.
\end{proof}
%
%
%
%
%
%

The following corollary will recur throughout the paper.
\begin{corollary}\label{beta lemma}
Let $c_0>-\ln3$. Then, there exists $\varepsilon>0$ and corresponding $C_{\varepsilon,n}, C_{\varepsilon,\ell}>0$ such that
for $|c-c_0|\leq\varepsilon$, $\beta$ satisfies
\begin{align*}
&(1)~ |\beta^{(n)}(c)|<C_{\varepsilon, n},	\cr
&(2)~ |\beta^{\prime}(c)|\geq C_{\varepsilon,\ell}.
\end{align*}
\end{corollary}
\begin{remark} 
This estimates on derivatives of $\beta$ show up too often throughout the paper, so we will not refer to this lemma except when it is necessary
to explicitly mention it.
\end{remark}
\begin{proof}
\noindent(1) By definition, $\beta(c)$ is infinitely differentiable with respect to $c$. Therefore, any derivatives of $\beta$ is continuous, and attain its maximum and minimum in the closed interval $|c-c_0|\leq\varepsilon$.\newline
\noindent (2) Take $\varepsilon$ sufficiently small so that any $c$ satisfying $|c-c_0|\leq\varepsilon$ still satisfies $c>-\ln 3$. Then, by Proposition \ref{betathm}, $\beta^{\prime}(c)$ is strictly negative. Therefore, the $|\beta^{\prime}(c)|$ is a continuous function that never vanishes, and we can find $C_{\varepsilon,\ell}>0$ such that
$|\beta^{\prime}(c)|\geq C_{\varepsilon,\ell}$ on the closed interval $|c-c_0|\leq\varepsilon$.
\end{proof}
\section{Linearization of Fermi-Dirac model}
In this section, we consider the linearization of the Fermi-Dirac distribution near a global Fermi-Dirac distribution:
\begin{align}\label{globalmax}
	m(p)=\frac{1}{e^{a_0|p|^2+c_0}+1},	
\end{align}
where $a_0$ and $c_0$ are determined by (\ref{by}) and $N_0$, $P_0$, $E_0$ satisfy ($P_0=0$)
\begin{align*}
\beta(-\ln3)>\frac{N_0}{E_0^{3/5}}.
\end{align*}
\subsection{Transitional Fermi-Dirac distribution} To study the linearization of the relaxation operator, we define the transitional local Fermi-Dirac distribution:
\begin{align}\label{Localtheta}
\mathcal{F}(\theta)=\frac{1}{e^{a_{\theta}\big|p-\frac{P_{\theta}}{N_{\theta}}\big|^2+c_{\theta}}+1},	
\end{align}
where $N_{\theta},P_{\theta},E_{\theta}$ denotes the transition of macroscopic fields from $(N,P,E)$ to $(N_0,P_0,E_0)$ ($0\leq \theta \leq 1$):
\begin{align*}
N_{\theta}=\theta N+(1-\theta)N_0,	\quad P_{\theta}&=\theta P,	\quad E_{\theta}=\theta E+(1-\theta)E_0,
\end{align*}
and $a_{\theta}$ and $c_{\theta}$ are defined by the following relations:
\begin{align}\label{recall ac}
\frac{N_{\theta}}{\left(E_{\theta}-\frac{P_{\theta}^2}{N_{\theta}}\right)^{\frac{3}{5}}} = \frac{\int_{\mathbb{R}^3}\frac{1}{e^{|p|^2+c_{\theta}}+1}dp}{\left(\int_{\mathbb{R}^3}\frac{|p|^2}{e^{|p|^2+c_{\theta}}+1}dp\right)^\frac{3}{5}},	\quad
a_{\theta}=\left(\int_{\mathbb{R}^3}\frac{1}{e^{|p|^2+c_{\theta}}+1}dp\right)^\frac{2}{3} N_{\theta}^{-\frac{2}{3}}.
\end{align}
Note that $\mathcal{F}(\theta)$ represents the transition from the global Fermi-Dirac $m(p)$ to the local Fermi-Dirac $\mathcal{F}(F)$:
\begin{align*}
\mathcal{F}(1)=\frac{1}{e^{a|p-\frac{P}{N}|^2+c}+1}, \quad \text{and} \quad  \mathcal{F}(0)=\frac{1}{e^{a_0|p|^2+c_0}+1}.
\end{align*}
\begin{definition} \label{Pfdef}
We define the macroscopic projection by
\begin{align*}
Pf\equiv \sum_{i=1}^5\langle f,e_i\rangle_{L^2_p} e_i,
\end{align*}
where $\{e_i\}_{1\leq i\leq5}$ is an orthonormal basis for the five dimensional linear space defined by
\begin{align}\label{orthobasis}
\begin{split}
e_1&= \frac{\sqrt{m-m^2}}{\sqrt{\displaystyle\int_{\mathbb{R}^3}m-m^2dp}},	\cr
e_i&= \frac{p_i\sqrt{m-m^2}}{\sqrt{\displaystyle\int_{\mathbb{R}^3}p_i^2(m-m^2)dp}} \quad i=2,3,4,	\cr
e_5&=\frac{|p|^2\sqrt{m-m^2}-\frac{\displaystyle\int_{\mathbb{R}^3}|p|^2(m-m^2)dp}{\displaystyle\int_{\mathbb{R}^3}m-m^2dp}\sqrt{m-m^2}}
{\sqrt{\displaystyle\int_{\mathbb{R}^3}\left(|p|^2\sqrt{m-m^2}-\frac{\displaystyle\int_{\mathbb{R}^3}|p|^2(m-m^2)dp}{\displaystyle\int_{\mathbb{R}^3}m-m^2dp}\sqrt{m-m^2}\right)^2dp}}.
\end{split}
\end{align}
\end{definition}
We now state the main goal of this section:
\begin{theorem}\label{Linearize} Assume $ c_{\theta} > -\ln3$. Then the local Fermi-Dirac distribution $\mathcal{F}(F)$ is linearized around a global Fermi-Dirac distribution $m$ if we define $F=m+\sqrt{m-m^2}f$ :
\begin{align*}
\mathcal{F}(F)=m+Pf\sqrt{m-m^2}+\sum_{1\leq i,j\leq5}\bigg\{\int_0^1 \{D^2_{(N_{\theta},P_{\theta},E_{\theta})}\mathcal{F}(\theta)\}_{i,j}(1-\theta)d\theta \bigg\} \langle f,e_i\rangle_{L^2_p}\langle f,e_j\rangle_{L^2_p}.
\end{align*}
\end{theorem}
We postpone the proof until various preliminary computations are completed. We start with the computations of the derivatives of transitional macroscopic fields.

\subsection{Derivatives of transitional macroscopic fields}
First we need the following lemma, which is frequently used throughout this subsection:
\begin{lemma}\label{ek-na} Define the constant $k$ by
\begin{align}\label{kdef}
k \equiv \int_{\mathbb{R}^3}(m-m^2)dp.
\end{align}
 Assume\begin{align*}
\frac{N_0}{E_0^{3/5}}<\beta(-\ln3).
\end{align*}
Then we have
\begin{align*}
E_0k-\frac{9N_0^2}{10a_0}>0.
\end{align*}
\end{lemma}
\begin{proof}
Note that the assumption guarantees that we can find $c_0>-\ln3$ by Theorem \ref{unique c}. An explicit computation gives
\begin{align*}
E_0k-\frac{9N_0^2}{10a_0}
&=\int_{\mathbb{R}^3}\frac{|p|^2}{e^{a_0|p|^2+c_0}+1}dp\int_{\mathbb{R}^3}\frac{e^{a_0|p|^2+c_0}}{(e^{a_0|p|^2+c_0}+1)^2}dp\cr
&-\frac{9}{10a_0}\left(\int_{\mathbb{R}^3}\frac{1}{e^{a_0|p|^2+c_0}+1}dp\right)^2.
\end{align*}
We then note from the proof of Proposition \ref{betathm} that the r.h.s is $-\frac{(4\pi)^2}{2a_0^4}B(c_0)$, the strict positiveness of which under the assumption $c_0>-\ln3$ is also shown in Proposition {\ref{betathm}}.
\end{proof}
\begin{lemma}\label{beta0} Assume $ c_{0}> -\ln 3$, then we have
\begin{align}\label{left}
\big(\beta^{-1}\big)^{\prime}\big(\beta(c_0)\big)&=\frac{E_0^\frac{8}{5}}{-E_0k+\frac{9N_0^2}{10a_0}},
\end{align}
\end{lemma}
\begin{proof} From the differentiation rule for composite functions, we have
\begin{align*}
\big(\beta^{-1}\big)^{\prime}(\beta(c)) = \frac{1}{\beta'(c)}.
\end{align*}

We then use (\ref{beta prime}) to get
\begin{align*}
\big(\beta^{-1}\big)^{\prime}(\beta(c)) &= \frac{\left(\int_{\mathbb{R}^3}\frac{|p|^2}{e^{|p|^2+c_0}+1}dp\right)^\frac{8}{5}}{\int_{\mathbb{R}^3}\frac{|p|^2}{e^{|p|^2+c_0}+1}dp\int_{\mathbb{R}^3}\frac{-e^{|p|^2+c_0}}{(e^{|p|^2+c_0}+1)^2}dp
-\frac{3}{5}\int_{\mathbb{R}^3}\frac{-|p|^2e^{|p|^2+c_0}}{(e^{|p|^2+c_0}+1)^2}dp\int_{\mathbb{R}^3}\frac{1}{e^{|p|^2+c_0}+1}dp}.
	\end{align*}
Writing it in the spherical coordinates and applying integration by parts:
\[
 u'=\frac{2re^{r^2+c_0}}{(e^{r^2+c_0}+1)^2}, \quad v=\frac{1}{2}r^3,
\]
and rewriting back in the Cartesian coordinate, we derive
\begin{align*}
\big(\beta^{-1}\big)^{\prime}(\beta(c))&= \frac{\left(\int_{\mathbb{R}^3}\frac{|p|^2}{e^{|p|^2+c_0}+1}dp\right)^\frac{8}{5}}{-\int_{\mathbb{R}^3}\frac{|p|^2}{e^{|p|^2+c_0}+1}dp\int_{\mathbb{R}^3}\frac{e^{|p|^2+c_0}}{(e^{|p|^2+c_0}+1)^2}dp
+\frac{9}{10}\left(\int_{\mathbb{R}^3}\frac{1}{e^{|p|^2+c_0}+1}dp\right)^2}.
\end{align*}
We then recall the definition of $N_0$, $P_0$, $E_0$ in (\ref{NPE0}) and observe
\begin{align}\label{k}
k = \int_{\mathbb{R}^3}\frac{e^{a_0|p|^2+c_0}}{(e^{a_0|p|^2+c_0}+1)^2}dp  = a_0^{-\frac{3}{2}}\int_{\mathbb{R}^3}\frac{e^{|p|^2+c_0}}{(e^{|p|^2+c_0}+1)^2}dp,
\end{align}
to obtain
\begin{align*}
\big(\beta^{-1}\big)^{\prime}(\beta(c))
=\frac{(a_0^{\frac{5}{2}}E_0)^\frac{8}{5}}{-a_0^{\frac{5}{2}}E_0\left(a_0^{\frac{3}{2}}k\right)+\frac{9}{10}\left(a_0^{\frac{3}{2}}N_0\right)^2}
=\frac{E_0^\frac{8}{5}}{-E_0k+\frac{9N_0^2}{10a_0}}.
\end{align*}
\end{proof}
Now we can calculate $\nabla_{(N_{\theta},P_{\theta},E_{\theta})}c_{\theta}$.
\begin{lemma}\label{diffc} Assume $c_{\theta}> -\ln 3$. Then we have
\begin{align*}
&(1) \ \frac{\partial c_\theta}{\partial N_\theta}\bigg|_{\theta=0} =\frac{E_0}{-E_0k+\frac{9N_0^2}{10a_0}},	\cr
&(2) \ \frac{\partial c_\theta}{\partial P_\theta}\bigg|_{\theta=0} =0,	\cr
&(3) \ \frac{\partial c_\theta}{\partial E_\theta}\bigg|_{\theta=0} =-\frac{3}{5}\frac{N_0}{-E_0k+\frac{9N_0^2}{10a_0}},
\end{align*}
where $k$ is defined by (\ref{kdef}).
\end{lemma}
\begin{proof}
Thanks to Proposition \ref{betathm}, the definition of $c_{\theta}$ in (\ref{recall ac}) and the assumption $c_{\theta}>-\ln3$, we can write
\begin{align}\label{cbeta}
c_{\theta}=\beta^{-1}\left( \frac{N_{\theta}}{\left(E_{\theta}-\frac{P_{\theta}^2}{N_{\theta}}\right)^{\frac{3}{5}}}\right).
\end{align}
We differentiate (\ref{cbeta}) w.r.t. $N_{\theta}$:
\begin{align*}
\frac{\partial c_\theta}{\partial N_\theta}\bigg|_{\theta = 0}
= \frac{1}{\beta'(c_\theta)} \frac{\partial}{\partial N_\theta}\left( \frac{N_\theta}{(E_\theta-\frac{P_\theta^2}{N_\theta})^{\frac{3}{5}}}\right)\bigg|_{\theta = 0} 	
= \frac{1}{\beta'(c_\theta)} \frac{(E_\theta-\frac{8}{5}\frac{P_\theta^2}{N_\theta})}{(E_\theta-\frac{P_\theta^2}{N_\theta})^{\frac{8}{5}}}\bigg|_{\theta = 0}.
\end{align*}
We then recall $P_0=0$ and Lemma \ref{beta0}:
\begin{align*}
\frac{\partial c_\theta}{\partial N_\theta}\bigg|_{\theta = 0}&=\frac{E_0^\frac{8}{5}}{-E_0k+\frac{9N_0^2}{10a_0}}\frac{1}{E_0^{\frac{3}{5}}}	
=\frac{E_0}{-E_0k+\frac{9N_0^2}{10a_0}}.
\end{align*}
A similar computation using $P_0=0$ gives
\begin{align*}
\frac{\partial c_\theta}{\partial P_\theta}\bigg|_{\theta = 0}
= \frac{1}{\beta'(c_\theta)} \frac{\partial}{\partial P_\theta}\left(\ \frac{N_\theta}{(E_\theta-\frac{P_\theta^2}{N_\theta})^{\frac{3}{5}}}\right)\bigg|_{\theta = 0}	
= \frac{1}{\beta'(c_\theta)} \frac{\frac{6}{5}P_\theta}{(E_\theta-\frac{P_\theta^2}{N_\theta})^{\frac{8}{5}}}\bigg|_{\theta = 0}	
=0,
\end{align*}
and
\begin{align*}
\frac{\partial c_\theta}{\partial E_\theta}\bigg|_{\theta = 0}
= \frac{1}{\beta'(c_\theta)} \frac{\partial}{\partial E_\theta}\left( \frac{N_\theta}{(E_\theta-\frac{P_\theta^2}{N_\theta})^{\frac{3}{5}}}\right)\bigg|_{\theta = 0}	
= \frac{1}{\beta'(c_\theta)}  \frac{-\frac{3}{5}N_\theta}{(E_\theta-\frac{P_\theta^2}{N_\theta})^{\frac{8}{5}}}\bigg|_{\theta = 0}	
=-\frac{3}{5}\frac{N_0}{-E_0k+\frac{9N_0^2}{10a_0}}.
\end{align*}
\end{proof}
We now compute the derivatives of $a_\theta$ with respect to the macroscopic fields.
\begin{lemma}\label{diffa}  We have
\begin{align*}
&(1) \ \frac{\partial a_\theta}{\partial N_\theta}\bigg|_{\theta = 0} =	-\frac{3}{5}\frac{N_0}{-E_0k+\frac{9N_0^2}{10a_0}}, \cr
&(2) \ \frac{\partial a_\theta}{\partial P_\theta}\bigg|_{\theta = 0} =	0, \cr
&(3) \ \frac{\partial a_\theta}{\partial E_\theta}\bigg|_{\theta = 0} =\frac{2}{5}\frac{a_0k}{-E_0k+\frac{9N_0^2}{10a_0}},
\end{align*}
where $k$ is defined by (\ref{kdef}).
\end{lemma}
\begin{proof} We recall
\begin{align*}
a_{\theta}=\left(\frac{\int_{\mathbb{R}^3}\frac{1}{e^{|p|^2+c_{\theta}}+1}dp}{N_{\theta}}\right)^\frac{2}{3},
\end{align*}
to compute
\begin{align*}
(1) ~ \frac{\partial a_\theta}{\partial N_\theta}\bigg|_{\theta = 0}
&= \frac{2}{3} \left(\frac{\int_{\mathbb{R}^3}\frac{1}{e^{|p|^2+c_\theta}+1}dp}{N_\theta}\right)^{-\frac{1}{3}}
\frac{\partial}{\partial N_\theta} \left(\frac{\int_{\mathbb{R}^3}\frac{1}{e^{|p|^2+c_\theta}+1}dp}{N_\theta}\right)\bigg|_{\theta = 0}	\cr
&= \frac{2}{3} \left(\frac{\int_{\mathbb{R}^3}\frac{1}{e^{|p|^2+c_0}+1}dp}{N_0}\right)^{-\frac{1}{3}}
\left(\frac{N_0\int_{\mathbb{R}^3}\frac{-e^{|p|^2+c_0}}{(e^{|p|^2+c_0}+1)^2}dp\frac{\partial c_\theta}{\partial N_\theta}\Big|_{\theta=0}-\int_{\mathbb{R}^3}\frac{1}{e^{|p|^2+c_0}+1}dp}{N_0^2}\right).
\end{align*}
We then employ Lemma \ref{diffc} (1) and (\ref{k}) to proceed further as
\begin{align*}
\frac{\partial a_\theta}{\partial N_\theta}\bigg|_{\theta = 0}
&= \frac{2}{3} \left(\frac{a_0^{\frac{3}{2}}N_0}{N_0}\right)^{-\frac{1}{3}}
\left(\frac{N_0(-a_0^{\frac{3}{2}}k)\frac{\partial c_\theta}{\partial N_\theta}\Big|_{\theta=0}-a_0^{\frac{3}{2}}N_0}{N_0^2}\right)	\cr
&= \frac{2}{3} a_0
\left(\frac{-k\frac{\partial c_\theta}{\partial N_\theta}\Big|_{\theta=0}-1}{N_0}\right)	\cr
&= \frac{2}{3} \frac{a_0}{N_0}\left(-k\frac{E_0}{-E_0k+\frac{9N_0^2}{10a_0}}-1\right)	\cr
&=-\frac{3}{5}\frac{N_0}{-E_0k+\frac{9N_0^2}{10a_0}}.
\end{align*}

(2)~ A similar computation using Lemma \ref{diffc} (2) and $P_{0}=0$  gives
\begin{align*}
\frac{\partial a_\theta}{\partial P_\theta}\bigg|_{\theta = 0}
&= \frac{2}{3} \left(\frac{\int_{\mathbb{R}^3}\frac{1}{e^{|p|^2+c_\theta}+1}dp}{N_\theta}\right)^{-\frac{1}{3}}
\frac{\partial}{\partial P_\theta } \left(\frac{\int_{\mathbb{R}^3}\frac{1}{e^{|p|^2+c_\theta}+1}dp}{N_\theta}\right)\bigg|_{\theta = 0}	\cr
&= \frac{2}{3} \left(\frac{\int_{\mathbb{R}^3}\frac{1}{e^{|p|^2+c_\theta}+1}dp}{N_\theta}\right)^{-\frac{1}{3}}
\left(\frac{\int_{\mathbb{R}^3}\frac{-e^{|p|^2+c_\theta}}{(e^{|p|^2+c_\theta}+1)^2}dp\frac{\partial c_\theta}{\partial P_\theta}}{N_\theta}\right)\bigg|_{\theta = 0}	\cr
&=\frac{2}{3} \left(\frac{\int_{\mathbb{R}^3}\frac{1}{e^{|p|^2+c_0}+1}dp}{N_0}\right)^{-\frac{1}{3}}
\left(\frac{\int_{\mathbb{R}^3}\frac{-e^{|p|^2+c_0}}{(e^{|p|^2+c_0}+1)^2}dp\frac{\partial c_\theta}{\partial P_\theta}\Big|_{\theta = 0}}{N_0}\right)	\cr
&=0.
\end{align*}

(3)~ We use Lemma \ref{diffc} (3) as
 \begin{align*}
\frac{\partial a_\theta}{\partial E_\theta}\bigg|_{\theta = 0}
&= \frac{2}{3} \left(\frac{\int_{\mathbb{R}^3}\frac{1}{e^{|p|^2+c_\theta}+1}dp}{N_\theta}\right)^{-\frac{1}{3}}
\frac{\partial}{\partial E_\theta} \left(\frac{\int_{\mathbb{R}^3}\frac{1}{e^{|p|^2+c_\theta}+1}dp}{N_\theta}\right)\bigg|_{\theta = 0}	\cr
&= \frac{2}{3} \left(\frac{\int_{\mathbb{R}^3}\frac{1}{e^{|p|^2+c_0}+1}dp}{N_0}\right)^{-\frac{1}{3}}
\left(\frac{\int_{\mathbb{R}^3}\frac{-e^{|p|^2+c_0}}{(e^{|p|^2+c_0}+1)^2}dp\frac{\partial c_\theta}{\partial E_\theta}\Big|_{\theta=0}}{N_0}\right)	\cr
&=\frac{2}{3}\left(\frac{a_0^{\frac{3}{2}}N_0}{N_0}\right)^{-\frac{1}{3}}\left(\frac{-a_0^{\frac{3}{2}}k\frac{\partial c_\theta}{\partial E_\theta}|_{\theta=0}}{N_0}\right)	\cr
&=-\frac{2}{3}\frac{a_0k}{N_0}\frac{\partial c_\theta}{\partial E_\theta}\Big|_{\theta=0}	\cr
&=-\frac{2}{3}\frac{a_0k}{N_0}\left(-\frac{3}{5}\frac{N_0}{-E_0k+\frac{9N_0^2}{10a_0}}\right)	\cr
&=\frac{2}{5}\frac{a_0k}{-E_0k+\frac{9N_0^2}{10a_0}}.
\end{align*}
\end{proof}

\subsection{Derivatives of $\mathcal{F}(\theta)$} We now turn to the derivatives of $\mathcal{F}(\theta)$.
\begin{lemma}\label{derivatives F} We have
\begin{align*}
(1)~  \frac{\partial \mathcal{F}(\theta)}{\partial N_{\theta}}\bigg|_{\theta=0} &=\left(\frac{3}{5}\frac{N_0}{-E_0k+\frac{9N_0^2}{10a_0}}|p|^2-\frac{E_0}{-E_0k+\frac{9N_0^2}{10a_0}}\right)(m-m^2),\cr
(2) ~ \frac{\partial \mathcal{F}(\theta)}{\partial P_{\theta}}\bigg|_{\theta=0}&=\frac{2a_0}{N_0}p(m-m^2),\cr
(3) ~\frac{\partial \mathcal{F}(\theta)}{\partial E_{\theta}}\bigg|_{\theta=0}&=\left(-\frac{2}{5}\frac{a_0k}{-E_0k+\frac{9N_0^2}{10a_0}}|p|^2+\frac{3}{5}\frac{N_0}{-E_0k+\frac{9N_0^2}{10a_0}} \right)(m-m^2).
\end{align*}
\end{lemma}

\begin{proof} All of these identities follows from similar arguments as in the previous cases using Lemma \ref{diffc} and Lemma \ref{diffa}:
\begin{align*}
\noindent(1) \ \frac{\partial \mathcal{F}(\theta)}{\partial N_{\theta}}\bigg|_{\theta=0} &= \frac{- \left\{ \frac{\partial a_\theta}{\partial N_\theta}\big|p-\frac{P_\theta}{N_\theta}\big|^2+a_\theta\frac{2P_\theta}{N_\theta^2}(p-\frac{P_\theta}{N_\theta})+\frac{\partial c_\theta}{\partial N_\theta} \right\}e^{a_\theta\big|p-\frac{P_\theta}{N_\theta}\big|^2+c_\theta}}{(e^{a_\theta\big|p-\frac{P_\theta}{N_\theta}\big|^2+c_\theta}+1)^2}\bigg|_{\theta=0}	\cr
&= - \left( \frac{\partial a_\theta}{\partial N_\theta}\bigg|_{\theta=0}|p|^2+\frac{\partial c_\theta}{\partial N_\theta}\bigg|_{\theta=0} \right)\frac{e^{a_0|p|^2+c_0}}{(e^{a_0|p|^2+c_0}+1)^2}\cr &=\left(\frac{3}{5}\frac{N_0}{-E_0k+\frac{9N_0^2}{10a_0}}|p|^2-\frac{E_0}{-E_0k+\frac{9N_0^2}{10a_0}}\right)(m-m^2).
\end{align*}

\begin{align*}
\noindent(2) \ \frac{\partial \mathcal{F}(\theta)}{\partial P_{\theta}}\bigg|_{\theta=0}&=\frac{-\left\{ \frac{\partial a_\theta}{\partial P_\theta}\big|p-\frac{P_\theta}{N_\theta}\big|^2-a_\theta\frac{2}{N_\theta}(p-\frac{P_\theta}{N_\theta})+\frac{\partial c_\theta}{\partial P_\theta} \right\}e^{a_\theta\big|p-\frac{P_\theta}{N_\theta}\big|^2+c_\theta}}{(e^{a_\theta\big|p-\frac{P_\theta}{N_\theta}\big|^2+c_\theta}+1)^2}\bigg|_{\theta=0}	\cr
&=-\left( \frac{\partial a_\theta}{\partial P_\theta}\bigg|_{\theta=0}|p|^2-a_0\frac{2}{N_0}p+\frac{\partial c_\theta}{\partial P_\theta}\bigg|_{\theta=0} \right)\frac{e^{a_0|p|^2+c_0}}{(e^{a_0|p|^2+c_0}+1)^2}\cr	
&=\frac{2a_0}{N_0}p(m-m^2).
\end{align*}

\begin{align*}
\hspace{-1.2cm}(3) \ \frac{\partial \mathcal{F}(\theta)}{\partial E_{\theta}}\bigg|_{\theta=0} &= \frac{-\left\{ \frac{\partial a_\theta}{\partial E_\theta}\big|p-\frac{P_\theta}{N_\theta}\big|^2+\frac{\partial c_\theta}{\partial E_\theta} \right\}e^{a_\theta\big|p-\frac{P_\theta}{N_\theta}\big|^2+c_\theta}}{(e^{a_\theta\big|p-\frac{P_\theta}{N_\theta}\big|^2+c_\theta}+1)^2}\bigg|_{\theta=0}	\cr
&= -\left(\frac{\partial a_\theta}{\partial E_\theta}\bigg|_{\theta=0}|p|^2+\frac{\partial c_\theta}{\partial E_\theta}\bigg|_{\theta=0} \right)\frac{e^{a_0|p|^2+c_0}}{(e^{a_0|p|^2+c_0}+1)^2}	\cr
&=\left(-\frac{2}{5}\frac{a_0k}{-E_0k+\frac{9N_0^2}{10a_0}}|p|^2+\frac{3}{5}\frac{N_0}{-E_0k+\frac{9N_0^2}{10a_0}} \right)(m-m^2).
\end{align*}
\end{proof}
%
%
%
%
\subsection{Proof of Theorem \ref{Linearize}}
Now we turn to the proof of Theorem \ref{Linearize}. Using Taylor's theorem around $\theta=0$, we obtain
\begin{align}\label{turn back to}
\mathcal{F}(1)=\mathcal{F}(0)+\mathcal{F}'(0)+\int_0^1\mathcal{F}''(\theta)(1-\theta)d\theta.
\end{align}
We know $\mathcal{F}(0)=m$. It remains to show for the  second and the third term in the r.h.s.\newline

\noindent(i) $\mathcal{F}'(0)$ : By chain rule, we have
\begin{align*}
\mathcal{F}'(0)&=\frac{d}{d\theta}\mathcal{F}(N_{\theta},P_{\theta},E_{\theta})|_{\theta=0}	\cr
&=\left(\frac{\partial N_{\theta}}{\partial \theta}\frac{\partial \mathcal{F}(\theta)}{\partial N_{\theta}}+\frac{\partial P_{\theta}}{\partial \theta}\frac{\partial \mathcal{F}(\theta)}{\partial P_{\theta}}+\frac{\partial E_{\theta}}{\partial \theta}\frac{\partial \mathcal{F}(\theta)}{\partial E_{\theta}}\right)\bigg|_{\theta=0}\cr
&=(N-N_0)\frac{\partial \mathcal{F}(\theta)}{\partial N_{\theta}}\bigg|_{\theta=0}+P\frac{\partial \mathcal{F}(\theta)}{\partial P_{\theta}}\bigg|_{\theta=0}+(E-E_0)\frac{\partial \mathcal{F}(\theta)}{\partial E_{\theta}}\bigg|_{\theta=0}.
\end{align*}
In the last line, we used $P_0=0$.
Then Lemma \ref{derivatives F}, together with
\begin{align*}
N-N_0&=\int_{\mathbb{R}^3}f\sqrt{m-m^2}dp,	\cr
P &= \int_{\mathbb{R}^3}fp\sqrt{m-m^2}dp,	\cr
E-E_0&=\int_{\mathbb{R}^3}f|p|^2\sqrt{m-m^2}dp,
\end{align*}
yields
\begin{align*}
\mathcal{F}'(0)
&= \int_{\mathbb{R}^3}f\sqrt{m-m^2}~dp\left(\frac{3}{5}\frac{N_0}{-E_0k+\frac{9N_0^2}{10a_0}}|p|^2-\frac{E_0}{-E_0k+\frac{9N_0^2}{10a_0}}\right)(m-m^2)	\cr
&+\int_{\mathbb{R}^3}fp\sqrt{m-m^2}dp\left(\frac{2a_0}{N_0}\right)p\,(m-m^2)\cr
&+\int_{\mathbb{R}^3}f|p|^2\sqrt{m-m^2}\,dp\left(-\frac{2}{5}\frac{a_0k}{-E_0k+\frac{9N_0^2}{10a_0}}|p|^2+\frac{3}{5}\frac{N_0}{-E_0k+\frac{9N_0^2}{10a_0}} \right)(m-m^2)	\cr
&= I+II+III\cr
&=\langle f,e_1\rangle_{L^2_p} e_1\sqrt{m-m^2}+II+\big(I+III-\langle f,e_1\rangle_{L^2_p} e_1\sqrt{m-m^2}~\big).
\end{align*}
We first show that $II$ is projection on $e_i$ $(i=2,3,4)$:
\begin{lemma}\label{II lemma}
We have
\begin{align}\label{II}
II=\sum_{i=2,3,4}\langle f,e_i\rangle_{L^2_p} e_i\sqrt{m-m^2}.
\end{align}
\end{lemma}
\begin{proof}
First, we make the following observation ($i=2,3,4$):
\begin{align*}
e_i&= \frac{p_i\sqrt{m-m^2}}{\sqrt{\int_{\mathbb{R}^3}p_i^2(m-m^2)dp}} 	\cr
&= \left(\int_{\mathbb{R}^3}p_i^2\frac{e^{a_0|p|^2+c_0}}{(e^{a_0|p|^2+c_0}+1)^2}dp\right)^{-\frac{1}{2}}p_i\sqrt{m-m^2}	\cr
&= \left(\int_{\mathbb{R}^3}\frac{1}{3}|p|^2\frac{e^{a_0|p|^2+c_0}}{(e^{a_0|p|^2+c_0}+1)^2}dp\right)^{-\frac{1}{2}}p_i\sqrt{m-m^2}	\cr
&= \left(\frac{1}{3}a_0^{-\frac{5}{2}}\int_{\mathbb{R}^3}\frac{|p|^2e^{|p|^2+c_0}}{(e^{|p|^2+c_0}+1)^2}dp\right)^{-\frac{1}{2}}p_i\sqrt{m-m^2}.
\end{align*}
We then recall the following identity obtained in the proof of Proposition \ref{betathm}:
\begin{align}\label{to derive}
\int_{\mathbb{R}^3}\frac{|p|^2e^{|p|^2+c_0}}{(e^{|p|^2+c_0}+1)^2}dp=\frac{3}{2}\int_{\mathbb{R}^3}\frac{1}{e^{|p|^2+c_0}+1}dp,
\end{align}
and derive from (\ref{by}) that
\begin{align}\label{N_0}
N_0=a_0^{-3/2}\int_{\mathbb{R}^3}\frac{1}{e^{|p|^2+c_0}+1}dp,
\end{align}
to get
\begin{align*}	
e_i&= \left(\frac{1}{2}a_0^{-\frac{5}{2}}\int_{\mathbb{R}^3}\frac{1}{e^{|p|^2+c_0}+1}dp\right)^{-\frac{1}{2}}p_i\sqrt{m-m^2}	\cr
&= \left(\frac{1}{2}\frac{N_0}{a_0}\right)^{-\frac{1}{2}}p_i\sqrt{m-m^2}	\cr
&= \left(\frac{2a_0}{N_0}\right)^{\frac{1}{2}}p_i\sqrt{m-m^2}.
\end{align*}
This immediately gives
\begin{align*}
II&=\left(\int_{\mathbb{R}^3}fp\sqrt{m-m^2}dp\right)\left(\frac{2a_0}{N_0}\right)p\,(m-m^2)\cr
&=\sum_{i=2,3,4}\left[\int_{\mathbb{R}^3}f\left\{\left(\frac{2a_0}{N_0}\right)^{\frac{1}{2}}p_i\sqrt{m-m^2}\right\}dp\right] \left\{\left(\frac{2a_0}{N_0}\right)^{\frac{1}{2}}p_i\sqrt{m-m^2}\right\}\sqrt{m-m^2}\cr
&=\sum_{i=2,3,4}\langle f,e_i\rangle_{L^2_p} e_i\sqrt{m-m^2}.
\end{align*}
\end{proof}
For the remaining terms,  we claim that
%
%
%
%
\begin{lemma}\label{I+III lemma} We have
\begin{align}\label{I+III}
I+III-\langle f,e_1\rangle_{L^2_p} e_1\sqrt{m-m^2}=\langle f,e_5 \rangle_{L^2_p} e_5\sqrt{m-m^2}.
\end{align}
\end{lemma}
\begin{proof}
We recall the definition of $k=\int_{\mathbb{R}^3}m-m^2dp$ to compute
\begin{align}\label{e1}
e_1= \frac{\sqrt{m-m^2}}{\sqrt{\int_{\mathbb{R}^3}m-m^2dp}}
= k^{-\frac{1}{2}}\sqrt{m-m^2}.
\end{align}
This readily yields
\begin{align}\label{fe1}
\langle f,e_1 \rangle_{L^2_p}e_1\sqrt{m-m^2} = \frac{1}{k}\left(\int_{\mathbb{R}^3}f\sqrt{m-m^2}dp\right)(m-m^2).
\end{align}
We now turn to the representation of $e_5$. First, we recall the definition of $k$ in (\ref{kdef}) and use (\ref{to derive}) and (\ref{N_0}) to compute the numerator of $e_5$ as follows:
\begin{align*}
&|p|^2\sqrt{m-m^2}-\frac{\int_{\mathbb{R}^3}|p|^2(m-m^2)dp}{\int_{\mathbb{R}^3}m-m^2dp}\sqrt{m-m^2}\cr
&\hspace{1cm}=\left(|p|^2-\frac{a_0^{-\frac{5}{2}}}{k}\int_{\mathbb{R}^3}\frac{|p|^2e^{|p|^2+c_0}}{(e^{|p|^2+c_0}+1)^2}dp\right)\sqrt{m-m^2}\cr
&\hspace{1cm}=\left(|p|^2-\frac{3a_0^{-\frac{3}{2}}}{2a_0k}\int_{\mathbb{R}^3}\frac{1}{e^{|p|^2+c_0}+1}dp\right)\sqrt{m-m^2}	\cr
&\hspace{1cm}=\left(|p|^2-\frac{3N_0}{2a_0k}\right)\sqrt{m-m^2}.
\end{align*}

For the computation of the denominator of $e_5$ (which we denote by $A$ for simplicity), we first write it using the definition of $k$ and $N_0$ above as
\begin{align}\label{inserting}
\begin{split}
A&=\int_{\mathbb{R}^3}
\left(
|p|^2\sqrt{m-m^2}-\frac{\int_{\mathbb{R}^3}|p|^2(m-m^2)dp}{\int_{\mathbb{R}^3}m-m^2dp}\sqrt{m-m^2}
\right)^2dp	\cr
&=\int_{\mathbb{R}^3}|p|^4(m-m^2)dp-\frac{1}{k}\left(\int_{\mathbb{R}^3}|p|^2(m-m^2)dp\right)^2\cr
&=I_1+I_2.
\end{split}
\end{align}
For $I_1$, we first observe that
\begin{align}\label{return}
\int_{\mathbb{R}^3}|p|^4(m-m^2)dp&=a_0^{-\frac{7}{2}}\int_{\mathbb{R}^3}\frac{|p|^4e^{|p|^2+c_0}}{(e^{|p|^2+c_0}+1)^2}dp.
\end{align}
We then write it in the spherical coordinate:
\[
\int_{\mathbb{R}^3}\frac{|p|^4e^{|p|^2+c_0}}{(e^{|p|^2+c_0}+1)^2}dp=\int_{0}^{2\pi}\int_{0}^{\pi}\int_{0}^{\infty}\frac{r^6e^{r^2+c_0}}{(e^{r^2+c_0}+1)^2}dr\sin\phi d\phi d\theta,
\]
and carry out the integration by parts:
\[
u'=\frac{2re^{r^2+c_0}}{(e^{r^2+c_0}+1)^2},\quad v= \frac{1}{2}r^5,
\]
to get
\[
\int_{0}^{\infty}\frac{r^6e^{r^2+c_0}}{(e^{r^2+c_0}+1)^2}dr=\frac{5}{2}\int_{0}^{\infty}\frac{r^4}{e^{r^2+c_0}+1}dr.
\]
We then go back to (\ref{return}) with these observations and find
\begin{align*}
\int_{\mathbb{R}^3}|p|^4(m-m^2)dp&=\frac{5}{2}a_0^{-\frac{7}{2}}\int_{0}^{2\pi}\int_{0}^{\pi}\int_{0}^{\infty}\frac{r^4}{e^{r^2+c_0}+1}dr\sin\phi d\phi d\theta	\cr
&=\frac{5}{2}a_0^{-\frac{7}{2}}\int_{\mathbb{R}^3}\frac{|p|^2}{e^{|p|^2+c_0}+1}dp	\cr
&=\frac{5E_0}{2a_0},
\end{align*}
where we used the definition of $E_0$:
\[
E_0=\int_{\mathbb{R}^3}m|p|^2dp=\int_{\mathbb{R}^3}\frac{|p|^2}{e^{a_0|p|^2+c_0}+1}dp=a_0^{-5/2}\int_{\mathbb{R}^3}\frac{|p|^2}{e^{|p|^2+c_0}+1}dp.
\]
For $I_2$, we use (\ref{to derive}) and (\ref{N_0}) to derive
\begin{align*}
\int_{\mathbb{R}^3}|p|^2(m-m^2)dp&=a_0^{-\frac{5}{2}}\int_{\mathbb{R}^3}\frac{|p|^2e^{|p|^2+c_0}}{(e^{|p|^2+c_0}+1)^2}dp\cr
&=\frac{3}{2}a_0^{-\frac{5}{2}}\int_{\mathbb{R}^3}\frac{1}{e^{|p|^2+c_0}+1}dp\cr
&=\frac{3N_0}{2a_0}.
\end{align*}
Inserting these computations to (\ref{inserting}), we get the following representation of the denominator of $e_5$:
\begin{align*}
A=\frac{5E_0}{2a_0}-\frac{9N_0^2}{4ka_0^2}.
\end{align*}
We now combine all the above identities for the denominator and the numerator of $e_5$ to obtain
\begin{align}\label{e5}
e_5&=\left(\frac{2}{5}\frac{a_0k}{E_0k-\frac{9N_0^2}{10a_0}}\right)^{\frac{1}{2}}\left(|p|^2-\frac{3N_0}{2a_0k}\right)\sqrt{m-m^2}.
\end{align}

Now, from (\ref{fe1}) and the definition of $I$, we first compute
\begin{align*}
&I-\langle f,e_1\rangle_{L^2_p} e_1\sqrt{m-m^2}	\cr
&=\left(\int_{\mathbb{R}^3}f\sqrt{m-m^2}dp\right)\left(\frac{3}{5}\frac{N_0}{-E_0k+\frac{9N_0^2}{10a_0}}|p|^2-\frac{E_0}{-E_0k+\frac{9N_0^2}{10a_0}}-\frac{1}{k}\right)(m-m^2)	\cr
&=\frac{2}{5}\frac{a_0k}{-E_0k+\frac{9N_0^2}{10a_0}}\left(\int_{\mathbb{R}^3}f\sqrt{m-m^2}dp\right)\left\{\frac{3N_0}{2a_0k}\left(|p|^2-\frac{3N_0}{2a_0k}\right)\right\}(m-m^2)\cr
&=\frac{2}{5}\frac{a_0k}{E_0k-\frac{9N_0^2}{10a_0}}
\left\{\int_{\mathbb{R}^3}f\sqrt{m-m^2}\Big(-\frac{3N_0}{2a_0k}\Big)dp\right\}
\left(|p|^2-\frac{3N_0}{2a_0k}\right)(m-m^2).
\end{align*}
Similarly $III$ also arranged as follows:
\begin{align*}
III&=\left(\int_{\mathbb{R}^3}f\sqrt{m-m^2}|p|^2dp\right)\left(-\frac{2}{5}\frac{a_0k}{-E_0k+\frac{9N_0^2}{10a_0}}|p|^2+\frac{3}{5}\frac{N_0}{-E_0k+\frac{9N_0^2}{10a_0}}\right)(m-m^2)	\cr
&=\frac{2}{5}\frac{a_0k}{E_0k-\frac{9N_0^2}{10a_0}}
\left\{\int_{\mathbb{R}^3}f\sqrt{m-m^2}|p|^2dp\right\}\left(|p|^2-\frac{3N_0}{2a_0k}\right)(m-m^2).
\end{align*}
Combining these identities gives
\begin{align*}
&\left\{I-\langle f,e_1\rangle_{L^2_p} e_1\sqrt{m-m^2}\right\}+III	\cr
&\qquad=\frac{2}{5}\frac{a_0k}{E_0k-\frac{9N_0^2}{10a_0}}
\left\{\int_{\mathbb{R}^3}f\sqrt{m-m^2}\left(|p|^2-\frac{3N_0}{2a_0k}\right)dp\right\}\left(|p|^2-\frac{3N_0}{2a_0k}\right)
(m-m^2)	\cr
&\qquad=\langle f,e_5\rangle_{L^2_p} e_5\sqrt{m-m^2},
\end{align*}
where we used (\ref{e5}).
\end{proof}
%
%
%
%
%
\subsection{Computation of the 2nd order term in Theorem \ref{Linearize}}: We now turn to the representation of the nonlinear terms:\newline
%
%
%
%
\noindent(ii) $\displaystyle\int_0^1\mathcal{F}''(\theta)(1-\theta)d\theta$ : By an explicit computation, we obtain

\begin{align*}
\mathcal{F}''(\theta)&= \frac{d^2\mathcal{F}}{d\theta^2}(N_{\theta},P_{\theta},E_{\theta})	\cr
&=(N-N_0,P,E-E_0)^TD^2_{(N_{\theta},P_{\theta},E_{\theta})}\mathcal{F}(\theta)(N-N_0,P,E-E_0)	\cr
&=\sum_{1\leq i,j\leq5} \{D^2_{(N_{\theta},P_{\theta},E_{\theta})}\mathcal{F}(\theta)\}_{i,j} \langle f,e_i \rangle_{L^2_p} \langle f,e_j\rangle_{L^2_p}.
\end{align*}

We then represent the second derivative as follows:
\begin{lemma}\label{diff2} There exists polynomial $P_{i,j}^{\mathcal{F}}$ and $R_{i,j}^{\mathcal{F}}$ satisfying following condition :
\begin{align*}
&\sum_{1\leq i,j\leq5} \{D^2_{(N_{\theta},P_{\theta},E_{\theta})}\mathcal{F}(\theta)\}_{i,j} \langle f,e_i \rangle_{L^2_p} \langle f,e_j \rangle_{L^2_p} \cr
&\hspace{1.5cm}=\sum_{1\leq i,j\leq5}\frac{P_{i,j}^{\mathcal{F}}(\theta)}{R_{i,j}^{\mathcal{F}}(\theta)}(\mathcal{F}(\theta)-\mathcal{F}(\theta)^2)\langle f,e_i\rangle_{L^2_p} \langle f,e_j\rangle_{L^2_p},
\end{align*}
where $P_{i,j}^{\mathcal{F}}(\theta)$ is a generically defined polynomial of
\[
N_{\theta}~,a_{\theta},~\Big(p-\frac{P_{\theta}}{N_{\theta}}\Big), ~D_{(N_{\theta},P_{\theta}, E_{\theta})}(a_{\theta},c_{\theta}),
~ D^2_{(N_{\theta}, P_{\theta},E_{\theta})}(a_{\theta}, c_{\theta}),~\mathcal{F}(\theta),~ 
\]
 and $R_{i,j}^{\mathcal{F}}(\theta)$ is generic polynomial of $N_{\theta}$, satisfying the following
 structural assumptions:\newline
\begin{itemize}
\item \ $ (\mathcal{H}_{\mathcal{F}}1)  P_{i,j}^{\mathcal{F}}$ is a polynomial such that $P_{i,j}^{\mathcal{F}}(0,0,...,0)=0$.
\item \ $ (\mathcal{H}_{\mathcal{F}}2)  R_{i,j}^{\mathcal{F}}$ is a monomial.
\end{itemize}
In other words, for a multi-index  $m=(m_1,m_2,...,m_n)$,
\begin{itemize}
\item \ $(\mathcal{H}_{\mathcal{F}}1)  P_{i,j}^{\mathcal{F}}(x_1,x_2,...,x_n)=\sum_{m}a_m x_1^{m_1}x_2^{m_2}\cdots x_n^{m_n}$, where $a_0=0$.
\item \ $(\mathcal{H}_{\mathcal{F}}2)  R_{i,j}^{\mathcal{F}}(x_1,x_2,...,x_n)=a_m x_1^{m_1}x_2^{m_2}\cdots x_n^{m_n}$.
\end{itemize}
\end{lemma}
\begin{proof}
We only consider the $(1,1)$ element of $D^2_{(N_{\theta},P_{\theta},E_{\theta})}\mathcal{F}(\theta)$, that is  $\frac{\partial^2 \mathcal{F}(\theta)}{\partial N_{\theta}^2}$. Other elements can be treated similarly. Thanks to Lemma \ref{derivatives F} (1), we have
\begin{align*}
\frac{\partial^2 \mathcal{F}(\theta)}{\partial N_{\theta}^2}
&=\frac{\partial }{\partial N_{\theta}}
\left(\frac{-\left\{ \frac{\partial a_{\theta}}{\partial N_{\theta}}|p-\frac{P_{\theta}}{N_{\theta}}|^2+a_{\theta}\frac{2P_{\theta}}{N_{\theta}^2}(p-\frac{P_{\theta}}{N_{\theta}})+\frac{\partial c_{\theta}}{\partial N_{\theta}} \right\}
e^{a_{\theta}|p-\frac{P_{\theta}}{N_{\theta}}|^2+c_{\theta}}}{(e^{a_{\theta}|p-\frac{P_{\theta}}{N_{\theta}}|^2+c_{\theta}}+1)^2}\right)	\cr
&=-\frac{\partial }{\partial N_{\theta}}
\left[ \frac{\partial a_{\theta}}{\partial N_{\theta}}\bigg|p-\frac{P_{\theta}}{N_{\theta}}\bigg|^2+a_{\theta}\frac{2P_{\theta}}{N_{\theta}^2}\left(p-\frac{P_{\theta}}{N_{\theta}}\right)+\frac{\partial c_{\theta}}{\partial N_{\theta}} \right]
(\mathcal{F}(\theta)-\mathcal{F}(\theta)^2)	\cr
& \quad -\left( \frac{\partial a_{\theta}}{\partial N_{\theta}}\bigg|p-\frac{P_{\theta}}{N_{\theta}}\bigg|^2+a_{\theta}\frac{2P_{\theta}}{N_{\theta}^2}\left(p-\frac{P_{\theta}}{N_{\theta}}\right)+\frac{\partial c_{\theta}}{\partial N_{\theta}} \right)
\frac{\partial }{\partial N_{\theta}}
\left[\mathcal{F}(\theta)-\mathcal{F}(\theta)^2\right]	\cr
&=I+II,
\end{align*}
where we used
\[
\frac{e^{a_{\theta}|p-\frac{P_{\theta}}{N_{\theta}}|^2+c_{\theta}}}{(e^{a_{\theta}|p-\frac{P_{\theta}}{N_{\theta}}|^2+c_{\theta}}+1)^2}=\mathcal{F}(\theta)-\mathcal{F}(\theta)^2.
\]
An explicit computation yields
\begin{align*}
I&=
-\left(\frac{\partial^2a_{\theta}}{\partial N_{\theta}^2}\bigg|p-\frac{P_{\theta}}{N_{\theta}}\bigg|^2 + \frac{\partial a_{\theta}}{\partial N_{\theta}}\frac{4P_{\theta}}{N_{\theta}^2}\left(p-\frac{P_{\theta}}{N_{\theta}}\right)-4a_{\theta}\frac{P_{\theta}}{N_{\theta}^3}\left(p-\frac{P_{\theta}}{N_{\theta}}\right)+2a_{\theta}\frac{P_{\theta}^2}{N_{\theta}^4}
+ \frac{\partial^2c_{\theta}}{\partial N_{\theta}^2}\right)	\cr
&\quad \times (\mathcal{F}(\theta)-\mathcal{F}(\theta)^2).
\end{align*}
Similarly,
\begin{align*}
II&=
-\left( \frac{\partial a_{\theta}}{\partial N_{\theta}}\bigg|p-\frac{P_{\theta}}{N_{\theta}}\bigg|^2+a_{\theta}\frac{2P_{\theta}}{N_{\theta}^2}\left(p-\frac{P_{\theta}}{N_{\theta}}\right)+\frac{\partial c_{\theta}}{\partial N_{\theta}} \right)
\left(\frac{\partial \mathcal{F}(\theta)}{\partial N_{\theta}}-2\mathcal{F}(\theta)\frac{\partial \mathcal{F}(\theta)}{\partial N_{\theta}}\right)	\cr
&=-\left( \frac{\partial a_{\theta}}{\partial N_{\theta}}\bigg|p-\frac{P_{\theta}}{N_{\theta}}\bigg|^2+a_{\theta}\frac{2P_{\theta}}{N_{\theta}^2}\left(p-\frac{P_{\theta}}{N_{\theta}}\right)+\frac{\partial c_{\theta}}{\partial N_{\theta}} \right)^2
(1-2\mathcal{F}(\theta))(\mathcal{F}(\theta)-\mathcal{F}(\theta)^2).
\end{align*}
Note that we have used Lemma \ref{derivatives F} (1) in the last line. Therefore, in view of the  definitions of $P_{i,j}^{\mathcal{F}}(\theta)$ and $R_{i,j}^{\mathcal{F}}(\theta)$, we can represent $\frac{\partial^2 \mathcal{F}(\theta)}{\partial N_{\theta}^2}$ as
\[
\frac{\partial^2 \mathcal{F}(\theta)}{\partial N_{\theta}^2}=\frac{P_{i,j}^{\mathcal{F}}(\theta)}{R_{i,j}^{\mathcal{F}}(\theta)}(\mathcal{F}(\theta)-\mathcal{F}(\theta)^2).
\]
This completes the proof of the $(1,1)$ elements. Others parts can be proved in a similar manner.
\end{proof}
%
%
%
%
%
\subsection{Linearization of the collision frequency} We now turn to the linearization of the collision frequency.
\begin{theorem}\label{LinearizeCol} Assume $ c_{\theta} > -\ln3$. Then collision frequency  is linearized  around $m$ as follows:
\begin{align*}
\frac{1}{\tau}
=1+\sum_{i=1}^5\left(\int^1_0\mathcal{C}_i(\theta)d\theta\right)\langle f,e_i\rangle_{L^2_p},
\end{align*}
where $\mathcal{C}_i(\theta)~(i=1,\cdots,5)$ are given by
\begin{align*}
\mathcal{C}_1(\theta)&= \sqrt{k}X(\theta)\cr
&+\frac{2}{3}\sqrt{k}Y(\theta) a_{\theta}^{-\frac{1}{2}}
\left(\int_{\mathbb{R}^3}\frac{-e^{|p|^2+c_{\theta}}}{(e^{|p|^2+c_{\theta}}+1)^2}dp\frac{1}{\beta'(c_{\theta})}\frac{(E_{\theta}-\frac{8}{5}\frac{P_{\theta}^2}{N_{\theta}}-\frac{9N_0}{10a_0k}N_{\theta})}{(E_{\theta}-\frac{P_{\theta}^2}{N_{\theta}})^{\frac{8}{5}}}-a_{\theta}^{\frac{3}{2}}\right)N_{\theta}^{-1}  ,	\cr
\mathcal{C}_i(\theta)&=\frac{2}{3}\left(\frac{N_0}{2a_0}\right)^{\frac{1}{2}}Y(\theta) a_{\theta}^{-\frac{1}{2}}
\left(\int_{\mathbb{R}^3}\frac{-e^{|p|^2+c_{\theta}}}{(e^{|p|^2+c_{\theta}}+1)^2}dp\frac{1}{\beta'(c_{\theta})} \frac{\frac{6}{5}P_{\theta,i-1}}{(E_{\theta}-\frac{P_{\theta}^2}{N_{\theta}})^{\frac{8}{5}}}\right)N_{\theta}^{-1},~\mbox{(for $i=2,3,4$),}  \cr
\mathcal{C}_5(\theta)&=\frac{2}{3}\left(\frac{2}{5}\frac{a_0k}{E_0k-\frac{9N_0^2}{10a_0}}\right)^{-\frac{1}{2}}Y(\theta) a_{\theta}^{-\frac{1}{2}}
\left(\int_{\mathbb{R}^3}\frac{-e^{|p|^2+c_{\theta}}}{(e^{|p|^2+c_{\theta}}+1)^2}dp\frac{1}{\beta'(c_{\theta})} \frac{-\frac{3}{5}N_{\theta}}{(E_{\theta}-\frac{P_{\theta}^2}{N_{\theta}})^{\frac{8}{5}}}\right)N_{\theta}^{-1} ,
\end{align*}
with
\begin{align*} X(\theta)&=P'(N_{\theta})\left(C_1a_{\theta}^n+C_2a_{\theta}^m+C_3\right),	\cr
Y(\theta)&=P(N_{\theta})\left(nC_1a_{\theta}^{n-1}+mC_2a_{\theta}^{m-1}\right).
\end{align*}
\end{theorem}
\begin{proof}
We recall
\begin{align*}
N_{\theta}=\theta N+(1-\theta)N_0,	\quad P_{\theta}&=\theta P,	\quad E_{\theta}=\theta E+(1-\theta)E_0,
\end{align*}
and define the transitional collision frequency as
\begin{align*}
g(\theta) = P(N_{\theta})\left(C_1a_{\theta}^n+C_2a_{\theta}^m+C_3\right)+C_4.
\end{align*}
Then we note that
\[
g(1)=\frac{1}{\tau}, \quad g(0)=\frac{1}{\tau_0}\equiv P(N_0)\left(C_1a_0^n+C_2a_0^m+C_3\right)+C_4.
\]
Without loss of generality, we set $\tau_0$ to be $1$ for simplicity. Applying Taylor expansion, we derive
\begin{align}\label{linig}
\begin{split}
g(1)&=g(0)+\int_{0}^{1}g'(\theta)d\theta	\cr
&=g(0)+\int_0^1(N-N_0,P,E-E_0)\cdot \left(\frac{\partial g(\theta) }{\partial N_{\theta}},\frac{\partial g(\theta) }{\partial P_{\theta}},\frac{\partial g(\theta) }{\partial E_{\theta}}\right) d\theta.
\end{split}
\end{align}
\noindent(i)
$g'(\theta)$: Explicit calculation using chain rule gives
\begin{align*}
g'(\theta)&=\left(\int_{\mathbb{R}^3}\sqrt{m-m^2}fdp\right)X(\theta)	\cr
&+\left(\int_{\mathbb{R}^3}\sqrt{m-m^2}fdp\right)Y(\theta) \frac{2}{3} a_{\theta}^{-\frac{1}{2}}
\left(\frac{\int_{\mathbb{R}^3}\frac{-e^{|p|^2+c_{\theta}}}{(e^{|p|^2+c_{\theta}}+1)^2}\frac{\partial c_{\theta}}{\partial N_{\theta}}dp-a_{\theta}^{\frac{3}{2}}}{N_{\theta}}\right) \cr
&+\left(\int_{\mathbb{R}^3}\sqrt{m-m^2}fpdp\right)Y(\theta) \frac{2}{3} a_{\theta}^{-\frac{1}{2}}
\left(\frac{\int_{\mathbb{R}^3}\frac{-e^{|p|^2+c_{\theta}}}{(e^{|p|^2+c_{\theta}}+1)^2}\frac{\partial c_{\theta}}{\partial P_{\theta}}dp}{N_{\theta}}\right) \cr
&+ \left(\int_{\mathbb{R}^3}\sqrt{m-m^2}f|p|^2dp\right)Y(\theta) \frac{2}{3} a_{\theta}^{-\frac{1}{2}}
\left(\frac{\int_{\mathbb{R}^3}\frac{-e^{|p|^2+c_{\theta}}}{(e^{|p|^2+c_{\theta}}+1)^2}\frac{\partial c_{\theta}}{\partial E_{\theta}}dp}{N_{\theta}}\right).
\end{align*}
Then, a tedious calculation using Lemma \ref{diffc} and Lemma \ref{diffa} together with (\ref{II}), (\ref{e1}) and (\ref{e5}) yields the desired result. We omit the details.
\end{proof}

\subsection{Linearized Quantum BGK model for fermions}
We employ the notation  $P_{i,j}^{\mathcal{F}}$ and $R_{i,j}^{\mathcal{F}}$ generically from now on, since, once the property $(\mathcal{H}_{\mathcal{F}}1)$ and $(\mathcal{H}_{\mathcal{F}}2)$ are satisfied, the exact form are not relevant. We also introduce the following three notations for notational simplicity:
\begin{align*}
\mathcal{Q}_{i,j}^{\mathcal{F}}(\theta)=\frac{P_{i,j}^{\mathcal{F}}(\theta)}{R_{i,j}^{\mathcal{F}}(\theta)},
\end{align*}
and
\begin{align*}
\mathcal{B}_{i,j}^{\mathcal{F}}=\int_0^1\mathcal{Q}_{i,j}^{\mathcal{F}}(\theta)\frac{(\mathcal{F}(\theta)-\mathcal{F}(\theta)^2)}{\sqrt{m-m^2}}(1-\theta)d\theta, \quad C_i^{\tau}=\int_0^1 C_i(\theta) d\theta.
\end{align*}
Now, we turn back to (\ref{turn back to}) with all these computations to get
\begin{align*}
\mathcal{F}(F)=m+\sqrt{m-m^2}Pf+\sqrt{m-m^2}\sum_{1\leq i,j\leq 5}\mathcal{B}_{i,j}^{\mathcal{F}}\langle f,e_i\rangle_{L^2_p} \langle f,e_j \rangle_{L^2_p}.
\end{align*}
\[
\frac{1}{\tau}=1+\sum_{i=1}^5\mathcal{C}^{\tau}_i\langle f,e_i\rangle_{L^2_p}.
\]
%
%
%
%
We summarize all the argument of this section so far in the following proposition.
\begin{proposition}
The relaxation collision operator is linearized around the global Fermi-Dirac distribution $m$ as follows:
\begin{align*}
&\frac{1}{\sqrt{m-m^2}}\frac{1}{\tau}\left\{\mathcal{F}(F)-F\right\}\cr
&\qquad=\left\{1+\sum_{i=1}^5\mathcal{C}^{\tau}_i\langle f,e_i\rangle_{L^2_p}\right\}
\left\{(Pf-f)+\sum_{1\leq i,j\leq 5}\mathcal{B}_{i,j}^{\mathcal{F}}\langle f,e_i\rangle_{L^2_p} \langle f,e_j\rangle_{L^2_p}\right\} .
\end{align*}
\end{proposition}
%
%
%
%
 We now substitute
\[
F=m+\sqrt{m-m^2}f,
\]
into (\ref{QBGK}) to obtain the perturbed Fermi-Dirac model:
\begin{align}\label{f}
\begin{split}
\partial_tf+p \cdot \nabla_xf &= Lf + \Gamma (f), \cr
f(x,p,0)&=f_0(x,p), 
\end{split}
\end{align}
where $f_0(x,p)=\frac{F_0(x,p)-m}{\sqrt{m-m^2}}$. The linearized relaxation operator $L$ and nonlinear perturbation term $\Gamma$ is defined as
\begin{align*}
Lf=Pf-f,
\end{align*}
and
\[
\Gamma(f)=\sum_{i=1}^3\Gamma_i(f),
\]
with
\begin{align}\label{Gamma}
\begin{split}
\Gamma_1(f)&=\sum_{1\leq i,j\leq 5}\mathcal{B}_{i,j}^{\mathcal{F}}\langle f,e_i\rangle_{L^2_p} \langle f,e_j \rangle_{L^2_p},\cr
\Gamma_2(f)&=\left\{\sum_{i=1}^5\mathcal{C}^{\tau}_i\langle f,e_i\rangle_{L^2_p}\right\}(Pf-f),\cr
\Gamma_3(f)&=\sum_{1\leq i,j,k\leq 5}\mathcal{B}_{i,j}^{\mathcal{F}}\mathcal{C}^{\tau}_k\langle f,e_i\rangle_{L^2_p}
\langle f,e_j\rangle_{L^2_p} \langle f,e_k\rangle_{L^2_p}.
\end{split}
\end{align}
Then the conservation laws (\ref{Conservation}) for $F$ now take the following form:
\begin{lemma}\label{consf} $f$ satisfies
\begin{align}\label{conservf}
\begin{split}
\int_{\mathbb{T}_x^3\times \mathbb{R}_p^3}f(x,p,t)\sqrt{m-m^2}dxdp&=\int_{\mathbb{T}_x^3\times \mathbb{R}_p^3}f_0(x,p)\sqrt{m-m^2}dxdp,	\cr
\int_{\mathbb{T}_x^3\times \mathbb{R}_p^3}f(x,p,t)p\sqrt{m-m^2}dxdp&=\int_{\mathbb{T}_x^3\times \mathbb{R}_p^3}f_0(x,p)p\sqrt{m-m^2}dxdp,	\cr
\int_{\mathbb{T}_x^3\times \mathbb{R}_p^3}f(x,p,t)|p|^2\sqrt{m-m^2}dxdp&=\int_{\mathbb{T}_x^3\times \mathbb{R}_p^3}f_0(x,p)|p|^2\sqrt{m-m^2}dxdp.
\end{split}
\end{align}
\end{lemma}
The following dissipative property of $L$ now follows from standard argument:
\begin{lemma}\label{coercivity}
Linearized relaxation operator $L$ satisfies the following coercivity property.
\begin{align*}
\langle Lf,f\rangle_{L_{x,p}^2}=-||(I-P)f||_{L_{x,p}^2}^2.
\end{align*}
\end{lemma}
\begin{proof}
Since $e_i$ $(i=1,\cdots,5)$ forms an orthonormal set by construction, $P$ is a orthogonal projection: $P^2=P$ and self-adjoint. Hence we have
\begin{align*}
\langle Pf,(I-P)f \rangle_{L_{p}^2}&=\langle Pf,f \rangle_{L_{p}^2}-\langle Pf,Pf \rangle_{L_{p}^2}	\cr
&=\langle Pf,f \rangle_{L_{p}^2}-\langle P^2f,f \rangle_{L_{p}^2}	\cr
&=0,
\end{align*}
which yields
\begin{align*}
\langle Lf,f\rangle_{L_{p}^2}&=\langle Pf-f,f\rangle_{L_{p}^2} =\langle Pf-f,-Pf+f\rangle_{L_{p}^2}
=-||(I-P)f||_{L_{p}^2}^2.
\end{align*}
\end{proof}

\section{Estimates on the nonlinear part}
In this section, we estimate the nonlinear part $\Gamma(f)$, which is crucial to close the energy estimate. For this, we first estimate $N$, $P$, $E$ and $a$ and $c$, when  $\mathcal{E}(t)$ is sufficiently small.
\subsection{Estimates on the macroscopic field} We start with the estimates of the macroscopic fields $N$, $P$ and $E$.
\begin{lemma}\label{esN} Suppose $\mathcal{E}(t)$ is sufficiently small, then we have the following estimates.
\begin{align*}
&(1) \ |N(x,t)-N_0|\leq C\sqrt{\mathcal{E}(t)},	\cr
&(2) \ |P(x,t)|\leq C\sqrt{\mathcal{E}(t)},	\cr
&(3) \ |E(x,t)-E_0|\leq C\sqrt{\mathcal{E}(t)},	\cr
&(4) \ \bigg|\,B\big(N,P,E\big)-\frac{N_0}{E_0^{3/5}}\bigg| \leq C\sqrt{\mathcal{E}(t)},
\end{align*}
for some constant $C>0$.
\end{lemma}
\begin{proof}
(1), (2) and (3) follows from a direct application of H\"{o}lder inequality. For example,
\begin{align*}
|E-E_0|&=\bigg|\int_{\mathbb{R}^3}|p|^2\sqrt{m-m^2}fdp\bigg| \leq \left(\int_{\mathbb{R}^3}f^2dp\right)^{\frac{1}{2}}\left(\int_{\mathbb{R}^3}|p|^4(m-m^2)dp\right)^{\frac{1}{2}} \leq C\sqrt{\mathcal{E}(t)}.
\end{align*}
We now turn to (4). Using the above estimate $(1)-(3)$, we get
\begin{align*}
E-\frac{P^2}{N} \geq E_0-C\sqrt{\mathcal{E}(t)}-\frac{C\mathcal{E}(t)}{N_0-C\sqrt{\mathcal{E}(t)}} \geq	E_0-C\sqrt{\mathcal{E}(t)},
\end{align*}
so that
\begin{align*}
\frac{N}{\left(E-\frac{P^2}{N}\right)^{\frac{3}{5}}}-\frac{N_0}{E_0^{\frac{3}{5}}} 
&\leq  \frac{N_0+C\sqrt{\mathcal{E}(t)}}{\left(E_0-C\sqrt{\mathcal{E}(t)}\right)^{\frac{3}{5}}}-\frac{N_0}{E_0^{\frac{3}{5}}}.
\end{align*}
Now, by mean value theorem, we can find  $E_0-C\sqrt{\mathcal{E}(t)}\leq k\leq E_0$ such that
\begin{align*}
\left(E_0-C\sqrt{\mathcal{E}(t)}\right)^{\frac{3}{5}} = E_0^{\frac{3}{5}}-\left(C\sqrt{\mathcal{E}(t)}\right)\frac{3}{5}k^{-\frac{2}{5}}\geq E_0^{\frac{3}{5}}-\frac{3}{5}C\sqrt{\mathcal{E}(t)}\left(E_0-C\sqrt{\mathcal{E}(t)}\right)^{-\frac{2}{5}}.
\end{align*}
Hence we have
\begin{align*}
\frac{N}{\left(E-\frac{P^2}{N}\right)^{\frac{3}{5}}}-\frac{N_0}{E_0^{\frac{3}{5}}}&\leq  \frac{N_0+C\sqrt{\mathcal{E}(t)}}{E_0^{\frac{3}{5}}-C\sqrt{\mathcal{E}(t)}}-\frac{N_0}{E_0^{\frac{3}{5}}}\cr
&\leq  \frac{C\big(E_0^{\frac{3}{5}}+N_0\big)\sqrt{\mathcal{E}(t)}}{\left(E_0^{\frac{3}{5}}-C\sqrt{\mathcal{E}(t)}\right)E_0^{\frac{3}{5}}}\cr
&\leq C\sqrt{\mathcal{E}(t)}.
\end{align*}
Lower bound can be obtained in a similar manner:
\begin{align*}
\frac{N}{\left(E-\frac{P^2}{N}\right)^{\frac{3}{5}}}-\frac{N_0}{E_0^{\frac{3}{5}}} &\geq  \frac{N_0-C\sqrt{\mathcal{E}(t)}}{\left(E_0+C\sqrt{\mathcal{E}(t)}\right)^{\frac{3}{5}}}-\frac{N_0}{E_0^{\frac{3}{5}}}	\cr
&\geq  \frac{N_0-C\sqrt{\mathcal{E}(t)}}{E_0^{\frac{3}{5}}+C\sqrt{\mathcal{E}(t)}}-\frac{N_0}{E_0^{\frac{3}{5}}}	\cr
&\geq  \frac{-C\sqrt{\mathcal{E}(t)}E_0^{\frac{3}{5}}-N_0C\sqrt{\mathcal{E}(t)}}{\left(E_0^{\frac{3}{5}}+C\sqrt{\mathcal{E}(t)}\right)E_0^{\frac{3}{5}}}	\cr
&\geq -C\sqrt{\mathcal{E}(t)}.
\end{align*}
\end{proof}
\begin{lemma}\label{esdN} Suppose $\mathcal{E}(t)$ is sufficiently small and $|\alpha|\geq1$, then we have
\begin{align*}
&(1) \ |\partial^{\alpha}N(x,t)| \leq C\sqrt{\mathcal{E}(t)},	\cr
&(2) \ |\partial^{\alpha}P(x,t)| \leq C\sqrt{\mathcal{E}(t)},	\cr
&(3) \ |\partial^{\alpha}E(x,t)| \leq C\sqrt{\mathcal{E}(t)},	\cr
&(4) \ \bigg|\partial^{\alpha}\left(\frac{P(x,t)}{N(x,t)}\right)\bigg| \leq C_{\alpha} \sqrt{\mathcal{E}(t)},	\cr
&(5) \ \left|\partial^{\alpha}B(N,P,E)\right|  \leq C_{\alpha}\sqrt{\mathcal{E}(t)},
\end{align*}
for some $C>0$ and $C_{\alpha}>0$.
\end{lemma}
\begin{proof}
(1)-(3) follows directly from applying $\partial^{\alpha}$ and estimating using H\"{o}lder inequality. For example, we have
\begin{align*}
|\partial^{\alpha}E|&=\bigg|\partial^{\alpha}\left(\int_{\mathbb{R}^3}|p|^2\left(m+\sqrt{m-m^2}f\right)dp\right)\bigg| \cr
&\leq \left(\int_{\mathbb{R}^3}|\partial^{\alpha}f|^2dp\right)^{\frac{1}{2}}\left(\int_{\mathbb{R}^3}|p|^4(m-m^2)dp\right)^{\frac{1}{2}} \cr
&\leq C\sqrt{\mathcal{E}(t)}.
\end{align*}
(4) A direct application of Leibniz rule and product rule of differentiation gives
\begin{align}\label{frac}
\begin{split}
\bigg|\partial^{\alpha}\left(\frac{P}{N}\right)\bigg|  &\leq C_{\alpha} \left(\sum_{|\alpha_1|\leq  |\alpha|}|\partial^{\alpha_1}P|\right)\left(\sum_{|\alpha_2|\leq |\alpha|}\bigg|\partial^{\alpha_2}\frac{1}{N}\bigg|\right)	\cr
&\leq C_{\alpha} \left(\sum_{|\alpha_1|\leq  |\alpha|}|\partial^{\alpha_1}P|\right)
\left(\sum_{0\leq n \leq  |\alpha|}\bigg|\frac{1}{N}\bigg|^{n+1}\right)\left(\sum_{1\leq|\alpha_2|\leq |\alpha|}|\partial^{\alpha_2}N|\right)^{|\alpha|}.
\end{split}
\end{align}
Then the desired result follows from the estimate (1),(2) of this lemma and Lemma \ref{esN} (1).
(5) Using chain rule, together with Lemma \ref{esN} and previous estimates in this lemma, the derivatives of denominator can be estimated as
\begin{align}\label{exp}
\begin{split}
\bigg|\partial^{\alpha}\left(E-\frac{P^2}{N}\right)^{\frac{3}{5}}\bigg| &\leq C_{\alpha}\left(\sum_{1\leq n\leq |\alpha|}\left(E-\frac{P^2}{N}\right)^{\frac{3}{5}-n}\right)
\left(\sum_{|\alpha_2|\leq |\alpha|}\bigg|\partial^{\alpha_2}\left(E-\frac{P^2}{N}\right)\bigg|\right)^{|\alpha|}	\cr
&\leq  C_{\alpha} \left(\sum_{1\leq n\leq |\alpha|}\left(E_0-C\sqrt{\mathcal{E}(t)}\right)^{\frac{3}{5}-n}\right) \left(C\sqrt{\mathcal{E}(t)}+C_{\alpha}\mathcal{E}(t)\right)^{|\alpha|}	\cr
&\leq C_{\alpha} \sqrt{\mathcal{E}(t)}.
\end{split}
\end{align}
Then the desired result follows directly from this and
\begin{align*}
|\partial^{\alpha}B\left(N,P,E\right)|  &\leq
C_{\alpha} \left(\sum_{|\alpha_1|\leq  |\alpha|}|\partial^{\alpha_1}N|\right)\left(\sum_{0\leq n \leq  |\alpha|}\bigg|\frac{1}{\left(E-\frac{P^2}{N}\right)^{\frac{3}{5}}}\bigg|^{n+1}\right) \cr
& \quad \times \left(\sum_{1\leq|\alpha_2|\leq |\alpha|}\bigg|\partial^{\alpha_2}\left(E-\frac{P^2}{N}\right)^{\frac{3}{5}}\bigg|\right)^{|\alpha|}.
\end{align*}
\end{proof}
%
%
%
%
\subsection{Estimates on the equilibrium coefficients} We now estimate the equilibrium coefficients $a$ and $c$.
\begin{lemma}\label{esac} Assume $\mathcal{E}(t)$ is sufficiently small. Then we have
\begin{align*}
&(1)~ |c(x,t)-c_0| \leq C\sqrt{\mathcal{E}(t)},	\cr
&(2)~ |a(x,t)-a_0| \leq C\sqrt{\mathcal{E}(t)},
\end{align*}
for some constant $C>0$.
\end{lemma}
\begin{proof}
(1) 
Since $\mathcal{E}(t)$ is sufficiently small, we have from Lemma \ref{esN} (4) that
\begin{align}\label{first}
0<\frac{N_0}{E_0^{3/5}}- C\sqrt{\mathcal{E}(t)}\leq\frac{N}{\left(E-\frac{P^2}{N}\right)^{3/5}} \leq \frac{N_0}{E_0^{3/5}}+ C\sqrt{\mathcal{E}(t)}<\beta(-\ln3),
\end{align}
so that, in view of Theorem \ref{unique c} and (\ref{beta inverse}), we can represent
\begin{align*}
c=\beta^{-1}\big(\,B(N,P,E)\,\big).
\end{align*}
We then deduce from the monotonicity of $\beta$ and \eqref{first} that
\begin{align}\label{ces}
\beta^{-1}\left(\frac{N_0}{E_0^{\frac{3}{5}}}+C\sqrt{\mathcal{E}(t)}\right) \leq c
\leq \beta^{-1}\left(\frac{N_0}{E_0^{\frac{3}{5}}}-C\sqrt{\mathcal{E}(t)} \right).	
\end{align}
Now, applying mean value theorem (which is possible due to Corollary \ref{beta lemma}) on both sides, we have
\begin{align*}
c\leq \beta^{-1}\left(\frac{N_0}{E_0^{\frac{3}{5}}}\right)-C\sqrt{\mathcal{E}(t)}\frac{1}{\beta'(k)} \quad \textit{for}
\leq c_0+C\sqrt{\mathcal{E}(t)},
\end{align*}
for some
$$
\beta^{-1}\left(\frac{N_0}{E_0^{\frac{3}{5}}}\right)<k<\beta^{-1}\left(\frac{N_0}{E_0^{\frac{3}{5}}}-C\sqrt{\mathcal{E}(t)}\right).
$$
Similarly, we have
\begin{align*}
c
\geq \beta^{-1}\left(\frac{N_0}{E_0^{\frac{3}{5}}}\right)+C\sqrt{\mathcal{E}(t)}\frac{1}{\beta'(k)}
\geq c_0-C\sqrt{\mathcal{E}(t)}.
\end{align*}
for some  $$\beta^{-1}\left(\frac{N_0}{E_0^{\frac{3}{5}}}+C\sqrt{\mathcal{E}(t)}\right)<k<\beta^{-1}\left(\frac{N_0}{E_0^{\frac{3}{5}}}\right)$$
Note that we have used Corollary \ref{beta lemma} (2) to bound $1/|\beta^{\prime}(k)|$.

\noindent(2) Thanks to the estimate (1) of this lemma and Lemma \ref{esN} (1), we estimate
\begin{align*}
a&=\left(\int_{\mathbb{R}^3}\frac{1}{e^{|p|^2+c}+1}dp\right)^\frac{2}{3}N^{-\frac{2}{3}} \leq \left(\int_{\mathbb{R}^3}\frac{1}{e^{|p|^2+c_0-C\sqrt{\mathcal{E}(t)}}+1}dp\right)^\frac{2}{3}(N_0-C\sqrt{\mathcal{E}(t)})^{-\frac{2}{3}}.
\end{align*}
Applying mean value theorem on
\begin{align*}
f(x)=\int_{\mathbb{R}^3}\frac{1}{e^{|p|^2+x}+1}dp,\quad \mbox{ and }\quad g(x)=x^{-2/3},
\end{align*}
yields
\begin{align*}
a&\leq \left(\int_{\mathbb{R}^3}\frac{1}{e^{|p|^2+c_0}+1}dp-C\sqrt{\mathcal{E}(t)}\int_{\mathbb{R}^3}\frac{-e^{|p|^2+k}}{(e^{|p|^2+k}+1)^2}dp\right)^\frac{2}{3}
\left(N_0^{-\frac{2}{3}}+\frac{2C}{3}\sqrt{\mathcal{E}(t)}h^{-\frac{5}{3}}\right),
\end{align*}
for $k\in(c_0-C\sqrt{\mathcal{E}(t)},c_0)$ and $h\in(N_0-C\sqrt{\mathcal{E}(t)},N_0)$. This gives, for sufficiently large $C$ and sufficiently small $\mathcal{E}(t)$
\begin{align*}
a&\leq \left(\int_{\mathbb{R}^3}\frac{1}{e^{|p|^2+c_0}+1}dp\right)^\frac{2}{3}N_0^{-\frac{2}{3}}+C\sqrt{\mathcal{E}(t)}= a_0+C\sqrt{\mathcal{E}(t)}.
\end{align*}
The estimate for lower bound is almost identical.
\end{proof}
We now turn to the estimates of derivatives of $a$ and $c$.
\begin{lemma}\label{esdc} Suppose $\mathcal{E}(t)$ is sufficiently small and $|\alpha|\geq1$. Then we have the following estimates for $c$.
\begin{align*}
(1) \ &|\partial^{\alpha}c| \leq C_{\alpha}\sqrt{\mathcal{E}(t)}, \cr
(2) \ &|\partial^{\alpha}\beta(c)| \leq C_{\alpha}\sqrt{\mathcal{E}(t)}, \cr
(3) \ &|\partial^{\alpha}(\nabla_{(N,P,E)}c)_{i}| \leq C_{\alpha}\sqrt{\mathcal{E}(t)} \quad \textit{for} \quad i=1,\cdots,5,	\cr
(4) \ &|\partial^{\alpha}(\nabla^2_{(N,P,E)}c)_{i,j}| \leq C_{\alpha}\sqrt{\mathcal{E}(t)}\quad \textit{for} \quad i,j=1,\cdots,5,
\end{align*}
for some $C_{\alpha}>0$.
\end{lemma}
\begin{proof}
(1)  Since
\[
(\beta^{-1})^{\prime}\left(B(N,P,E)\right)=\frac{1}{\beta^{\prime}(c)}.
\]
We easily see that $(\beta^{-1})^{(n)}$ takes the following form:
\[
(\beta^{-1})^{(n)}\left(B(N,P,E)\right)=\frac{P\big(\beta(c), \beta^{\prime}(c),\cdots,\beta^{(n)}(c)\big)}{|\beta^{\prime}(c)|^n},
\]
for some generic polynomial $P$ satisfying $P(0,0,\cdots,0)=0$. Therefore, Corollary \ref{beta lemma}, Lemma \ref{esN} (4) and Lemma \ref{esac} (1) give the following uniform bound
\[
\left|(\beta^{-1})^{(n)}B(N,P,E)\right|\leq C_n,
\]
for some $C_n>0$.
Then, the desired result follows from this, together with Lemma \ref{esdN} (5), and the following computation:
\begin{align*}
|\partial^{\alpha}c| &= \partial^{\alpha}\left\{\beta^{-1}\left(B(N,P,E)\right)\right\}	\cr
&\leq C_{\alpha}\sum_{n\leq|\alpha|}
\left(\bigg|\left(\beta^{-1}\right)^{(n)}B(N,P,E)\bigg|\right)	
\left(\sum_{1\leq|\alpha_1|\leq |\alpha|}\bigg|\partial^{\alpha_1}B(N,P,E)\bigg|\right)^{|\alpha|}.
\end{align*}

\noindent(2) Estimate on the derivative of $c$ above and Corollary \ref{beta lemma} readily gives
\begin{align*}
|\partial^{\alpha}\beta(c)| &\leq
C_{\alpha}\sum_{n\leq|\alpha|}
\left(|\beta^{(n)}(c)|\right)
\left(\sum_{1\leq|\alpha_1|\leq |\alpha|}|\partial^{\alpha_1}c|\right)^{|\alpha|}	
\leq C_{\alpha}\sqrt{\mathcal{E}(t)}.
\end{align*}
\noindent(3) We will consider the derivatives of $\partial c/\partial N$. We recall from Lemma \ref{diffc} (1) that
\begin{align*}
\frac{\partial c}{\partial N} &= \frac{1}{\beta'(c)}  \frac{E-\frac{8}{5}\frac{P^2}{N}}{(E-\frac{P^2}{N})^{\frac{8}{5}}}.
\end{align*}
Take $\partial^{\alpha}$, then we obtain
\begin{align}\label{p alpha}
\begin{split}
\bigg|\partial^{\alpha}\frac{\partial c}{\partial N}\bigg|
&\leq C_{\alpha }\sum_{|\alpha_1|+|\alpha_2| \leq|\alpha|}\bigg|\partial^{\alpha_1}\left( \frac{1}{\beta'(c)}\right)\bigg|  \bigg|\partial^{\alpha_2} \left(\frac{E-\frac{8}{5}\frac{P^2}{N}}{\left(E-\frac{P^2}{N}\right)^{\frac{8}{5}}}\right)\bigg|.
\end{split}
\end{align}
Employing the estimate (1) of this lemma and Corollary \ref{beta lemma}, we can estimate
\begin{align*}
\bigg|\partial^{\alpha} \left(\frac{1}{\beta'(c)}\right)\bigg|
&\leq C_{\alpha}
\left(\sum_{0\leq n \leq  |\alpha|}\bigg|\frac{1}{\left(\beta'(c)\right)^{n+1}}\bigg|\right)\left(\sum_{1\leq|\alpha_1|\leq |\alpha|}\big|\partial^{\alpha_1}\left\{\beta'(c)\right\}\big|\right)^{|\alpha|}	\cr
&\leq C_{\alpha}
\left(\sum_{0\leq n \leq  |\alpha|}\bigg|\frac{1}{\left(\beta'(c)\right)^{n+1}}\bigg|\right)\left(\sum_{1\leq|\alpha_1|\leq |\alpha|}\left|\beta^{(1+|\alpha_1|)}(c)\right|\left(\sum_{|\alpha_1|\leq|\alpha|}|\partial^{\alpha_1}c|\right)^{|\alpha|}\right)^{|\alpha|}	\cr
& \leq C_{\alpha} \sqrt{\mathcal{E}(t)}.
\end{align*}
On the other hand, by an almost identical manner as in the proof of Lemma \ref{esdN} (5), we can derive
\begin{align*}
\bigg|\partial^{\alpha} \left(\frac{E-\frac{8}{5}\frac{P^2}{N}}{\left(E-\frac{P^2}{N}\right)^{\frac{8}{5}}}\right)\bigg|
\leq C_{\alpha}\sqrt{\mathcal{E}(t)}.
\end{align*}
Inserting these estimates into (\ref{p alpha}) gives the desired result.\newline
\noindent(4) We only consider $(1,1)$ elements of $\nabla^2_{(N,P,E)}c$, which is
\begin{align*}
\frac{\partial^2c}{\partial N^2}= \frac{-\beta''(c)}{(\beta'(c))^3}  \frac{\left(E-\frac{8}{5}\frac{P^2}{N}\right)^2}{\left(E-\frac{P^2}{N}\right)^{\frac{16}{5}}}+ \frac{24}{25}  \frac{1}{\beta'(c)} \frac{P^4}{N^3}  \frac{1}{\left(E-\frac{P^2}{N}\right)^{\frac{13}{5}}}.
\end{align*}
Therefore, we can bound it by  $C\sqrt{\mathcal{E}(t)}$ similarly as in the proof of (3).
\end{proof}
\begin{lemma}\label{esda} Suppose $\mathcal{E}(t)$ is sufficiently small and $|\alpha|\geq1$. Then we have the following estimates for $a$.
\begin{align*}
(1) \ &|\partial^{\alpha}a| \leq C_{\alpha}\sqrt{\mathcal{E}(t)},	\cr
(2) \ &|\partial^{\alpha}(\nabla_{(N,P,E)}a)_{i}| \leq C_{\alpha}\sqrt{\mathcal{E}(t)} \quad \textit{for} \quad i=1,\cdots,5,	\cr
(3) \ &|\partial^{\alpha}(\nabla^2_{(N,P,E)}a)_{i,j}| \leq C_{\alpha}\sqrt{\mathcal{E}(t)}\quad \textit{for} \quad i,j=1,\cdots,5,
\end{align*}
for some $C_{\alpha}>0$.
\end{lemma}
\begin{proof}
We only consider $\partial^{\alpha}\left(\partial^2 a/\partial N^2\right)$. Recall the definition of $a$:
\begin{align*}
a=\left(\int_{\mathbb{R}^3}\frac{1}{e^{|p|^2+c}+1}dp\right)^\frac{2}{3}N^{-\frac{2}{3}}.
\end{align*}
Then explicit calculations give
\begin{align*}
\frac{\partial a}{\partial N}
&= \frac{2}{3} \left(\frac{\int_{\mathbb{R}^3}\frac{1}{e^{|p|^2+c}+1}dp}{N}\right)^{-\frac{1}{3}}
\left(\frac{N\int_{\mathbb{R}^3}\frac{-e^{|p|^2+c}}{(e^{|p|^2+c}+1)^2}\frac{\partial c}{\partial N}dp-\int_{\mathbb{R}^3}\frac{1}{e^{|p|^2+c}+1}dp}{N^2}\right).
\end{align*}
and
\begin{align*}
\frac{\partial^2a}{\partial N^2}
&=  -\frac{2}{9}
\left(\frac{\int_{\mathbb{R}^3}\frac{1}{e^{|p|^2+c}+1}dp}{N}\right)^{-\frac{4}{3}}
\left(\frac{N\int_{\mathbb{R}^3}\frac{-e^{|p|^2+c}}{(e^{|p|^2+c}+1)^2}\frac{\partial c}{\partial N}dp-\int_{\mathbb{R}^3}\frac{1}{e^{|p|^2+c}+1}dp}{N^2}\right)^2 \cr
& \quad +
\frac{2}{3} \left(\frac{\int_{\mathbb{R}^3}\frac{1}{e^{|p|^2+c}+1}dp}{N}\right)^{-\frac{1}{3}}\frac{1}{N^2}\frac{\partial c}{\partial N}\int_{\mathbb{R}^3}\frac{-e^{|p|^2+c}}{(e^{|p|^2+c}+1)^2}dp
\cr
&\quad +\frac{2}{3} \left(\frac{\int_{\mathbb{R}^3}\frac{1}{e^{|p|^2+c}+1}dp}{N}\right)^{-\frac{1}{3}}\frac{1}{N}\frac{\partial^2 c}{\partial N^2}\int_{\mathbb{R}^3}\frac{-e^{|p|^2+c}}{(e^{|p|^2+c}+1)^2}dp	\cr &\quad +\frac{2}{3} \left(\frac{\int_{\mathbb{R}^3}\frac{1}{e^{|p|^2+c}+1}dp}{N}\right)^{-\frac{1}{3}}\frac{1}{N}\left(\frac{\partial c}{\partial N}\right)^2
\int_{\mathbb{R}^3}\frac{e^{|p|^2+c}(e^{|p|^2+c}-1)}{(e^{|p|^2+c}+1)^3}dp	\cr
&\quad -\frac{2}{3} \left(\frac{\int_{\mathbb{R}^3}\frac{1}{e^{|p|^2+c}+1}dp}{N}\right)^{-\frac{1}{3}}\frac{1}{N^2}\frac{\partial c}{\partial N}\int_{\mathbb{R}^3}\frac{-e^{|p|^2+c}}{(e^{|p|^2+c}+1)^2}dp	\cr
& \quad -\frac{2}{N^2}\int_{\mathbb{R}^3}\frac{-e^{|p|^2+c}}{(e^{|p|^2+c}+1)^2}\frac{\partial c}{\partial N}dp+\frac{2}{N^3}\int_{\mathbb{R}^3}\frac{1}{e^{|p|^2+c}+1}dp.
\end{align*}
Therefore, the desired estimate is derived once we obtain the estimates for the derivatives of
\[
h(c)=\int_{\mathbb{R}^3}\frac{1}{e^{|p|^2+c}+1}dp,\quad k(c)=\int_{\mathbb{R}^3}\frac{-e^{|p|^2+c}}{(e^{|p|^2+c}+1)^2}dp.
\]
Then, it can be easily verified through an explicit computation that
\begin{align*}
|h^{(n)}(c)|,~|k^{(n)}(c)|  \leq C \int_{\mathbb{R}^3}\frac{1}{e^{|p|^2+c}+1}dp,
\end{align*}
which, thanks to the chain rule and Lemma \ref{esdc} (1), leads to
\begin{align*}
|\partial^{\alpha}h(c)|,~|\partial^{\alpha}k(c)|&\leq C_{\alpha}\sqrt{\mathcal{E}(t)}.
\end{align*}
This, together with  Lemma \ref{esN}-\ref{esdc}, gives the desired result.

\end{proof}
%
%
%
%
%
\section{Local existence}
\subsection{Estimates on the nonlinear term} Using the estimates for the macroscopic fields $(N,P,E)$ and the equilibrium coefficients $(a,c)$ in the previous section, we derive the following  estimate of the nonlinear terms:
\begin{proposition}\label{prop} Suppose $\mathcal{E}(t)$ is sufficiently small enough to satisfies  Lemma \ref{esN} - \ref{esda}. Then we have
\begin{align*}
\ \bigg|\int_{\mathbb{R}^3}\partial_{\beta}^{\alpha}\Gamma(f)gdp\bigg| &\leq C \!\!\!\sum_{|\alpha_1|+|\alpha_2|\leq|\alpha|}||\partial^{\alpha_1}f||_{L_p^2}||\partial^{\alpha_2}f||_{L_p^2}||g||_{L_p^2}\cr
&+ C \!\!\!\!\sum_{|\alpha_1|+|\alpha_2|+|\alpha_3|\leq|\alpha|}||\partial^{\alpha_1}f||_{L_p^2}||\partial^{\alpha_2}f||_{L_p^2}||\partial^{\alpha_3}f||_{L_p^2}||g||_{L_p^2}.
\end{align*}
\end{proposition}
\begin{proof}
We only consider $\Gamma_1$. Estimates for other terms are almost identical. Recall
\begin{align*}
\Gamma_1(f)=\sum_{1\leq i,j\leq 5}\left\{\int_0^1\mathcal{Q}_{i,j}^{\mathcal{F}}(\theta)\frac{(\mathcal{F}(\theta)-\mathcal{F}(\theta)^2)}{\sqrt{m-m^2}}(1-\theta)d\theta\right\}\langle f,e_i\rangle_{L^2_p} \langle f,e_j \rangle_{L^2_p}.
\end{align*}
We first claim the following:\newline
\noindent $\bullet$ {\bf Claim:} For sufficiently small $\mathcal{E}(t)$, we obtain
\begin{align*}
\bigg|\partial_{\beta}^{\alpha}\left\{\mathcal{Q}_{i,j}^{\mathcal{F}}(\theta)\frac{\mathcal{F}(\theta)-\mathcal{F}(\theta)^2}{\sqrt{m-m^2}}\right\}\bigg| \leq C_{\alpha,\beta} e^{-\frac{a_0}{8}|p|^2},
\end{align*}
for some $C_{\alpha,\beta}>0$. \newline

\noindent {\bf Proof of the claim:}
By Leibniz's rule
\begin{align*}
\bigg|\partial_{\beta}^{\alpha}\left\{\mathcal{Q}_{i,j}^{\mathcal{F}}(\theta)\frac{\mathcal{F}(\theta)-\mathcal{F}(\theta)^2}{\sqrt{m-m^2}}\right\}\bigg| &=C_{\alpha,\beta} \sum_{\substack{|\alpha_1|+|\alpha_2|\leq|\alpha|\cr |\beta_1|+|\beta_2|\leq |\beta|}}\big|\partial_{\beta_1}^{\alpha_1}\left\{\mathcal{Q}_{i,j}^{\mathcal{F}}\right\}\big|\bigg|\partial_{\beta_2}^{\alpha_2}\left\{\frac{\mathcal{F}(\theta)-\mathcal{F}(\theta)^2}{\sqrt{m-m^2}}\right\}\bigg|.
\end{align*}
The uniform bound of $\big|\partial_{\beta_1}^{\alpha_1}\left\{\mathcal{Q}_{i,j}^{\mathcal{F}}\right\}\big|\leq C_N$ is rather straightforward (and tedious) from
the definition and all the upper and lower bound estimates for the equilibrium coefficients and conservative quantities in the previous section.
For the remaining part, we observe from
\[
\frac{1}{\sqrt{m-m^2}}=e^{\frac{a_0}{2}|p|^2+\frac{c_0}{2}}+e^{-\frac{a_0}{2}|p|^2-\frac{c_0}{2}},
\]
that
\begin{align*}
&\bigg|\partial_{\beta}^{\alpha}\left\{\frac{\mathcal{F}(\theta)-\mathcal{F}(\theta)^2}{\sqrt{m-m^2}}\right\}\bigg|\cr
&\leq C_{\alpha,\beta} \sum_{\substack{|\beta_1|+|\beta_2|\leq |\beta|}}\bigg|\partial_{\beta_1}\left(e^{\frac{a_0}{2}|p|^2+\frac{c_0}{2}}+e^{-\frac{a_0}{2}|p|^2-\frac{c_0}{2}}\right)\bigg|  ~\bigg|\partial_{\beta_2}^{\alpha}\left(\frac{e^{a_{\theta}\big|p-\frac{P_{\theta}}{N_{\theta}}\big|^2+c_{\theta}}}{\left(e^{a_{\theta}\big|p-\frac{P_{\theta}}{N_{\theta}}\big|^2+c_{\theta}}+1\right)^2}\right)\bigg|.	
\end{align*}
By a simple calculation, we get
\begin{align*}
\bigg|\partial_{\beta}\left(e^{\frac{a_0}{2}|p|^2+\frac{c_0}{2}}+e^{-\frac{a_0}{2}|p|^2-\frac{c_0}{2}}\right)\bigg|
&\leq \bigg|P_{\beta}(a_0,p)\left(e^{\frac{a_0}{2}|p|^2+\frac{c_0}{2}}+e^{-\frac{a_0}{2}|p|^2-\frac{c_0}{2}}\right)\bigg|.
\end{align*}
 and
\begin{align*}
\bigg|\partial_{\beta}^{\alpha}\left(\frac{e^{a_{\theta}\big|p-\frac{P_{\theta}}{N_{\theta}}\big|^2+c_{\theta}}}{\left(e^{a_{\theta}\big|p-\frac{P_{\theta}}{N_{\theta}}\big|^2+c_{\theta}}+1\right)^2}\right)\bigg|
&\leq \bigg|P_{\alpha,\beta}\left(\partial^{\alpha} a_{\theta},\partial^{\alpha} c_{\theta},\partial^{\alpha}_{\beta} \left(p-\frac{P_{\theta}}{N_{\theta}}\right)\right)
\frac{1}{e^{a_{\theta}\big|p-\frac{P_{\theta}}{N_{\theta}}\big|^2+c_{\theta}}+1}\bigg|\cr
&\leq \bigg|P_{\alpha,\beta}\left(\partial^{\alpha} a_{\theta},\partial^{\alpha} c_{\theta},\partial^{\alpha}\frac{P_{\theta}}{N_{\theta}}\right)
\frac{1}{e^{a_{\theta}\big|p-\frac{P_{\theta}}{N_{\theta}}\big|^2+c_{\theta}}+1}\bigg|,
\end{align*}
where $P_{\beta}$ and $P_{\alpha,\beta}$ denote generically defined polynomials. These estimates and the lower and upper bounds established
in the previous  section on the equilibrium coefficients give
\begin{align*}
\bigg|\partial_{\beta}^{\alpha}\left\{\frac{\mathcal{F}(\theta)-\mathcal{F}(\theta)^2}{\sqrt{m-m^2}}\right\}\bigg|
\leq C_{\alpha,\beta}
 \frac{e^{\frac{a_0}{2}|p|^2+\frac{c_0}{2}}+e^{-\frac{a_0}{2}|p|^2-\frac{c_0}{2}}}{e^{a_{\theta}\big|p-\frac{P_{\theta}}{N_{\theta}}\big|^2+c_{\theta}}+1}.
\end{align*}
Finally, the desired estimate follows from the following computation:
\begin{align*}
\frac{e^{\frac{a_0}{2}|p|^2+\frac{c_0}{2}}+e^{-\frac{a_0}{2}|p|^2-\frac{c_0}{2}}}{e^{a_{\theta}\big|p-\frac{P_{\theta}}{N_{\theta}}\big|^2+c_{\theta}}+1}	&\leq
\frac{e^{\frac{a_0}{2}|p|^2+\frac{c_0}{2}}+e^{-\frac{a_0}{2}|p|^2-\frac{c_0}{2}}}{e^{a_{\theta}\left(\frac{3}{4}|p|^2-3\frac{P_{\theta}^2}{N_{\theta}^2}\right)+c_{\theta}}} \cr
&\leq
\frac{e^{\frac{a_0}{2}|p|^2+\frac{c_0}{2}}+e^{-\frac{a_0}{2}|p|^2-\frac{c_0}{2}}}{e^{\left(a_0-C\sqrt{\mathcal{E}(t)}\right)\left(\frac{3}{4}|p|^2-C\sqrt{\mathcal{E}(t)}\right)+c_0-C\sqrt{\mathcal{E}(t)}}} \cr
&\leq e^{\left(-\frac{a_{0}}{4}+C\sqrt{\mathcal{E}(t)}\right)|p|^2-\frac{c_{0}}{2}+C\sqrt{\mathcal{E}(t)}}+e^{\left(-\frac{3a_{0}}{4}+C\sqrt{\mathcal{E}(t)}\right)|p|^2-\frac{3c_{0}}{2}+C\sqrt{\mathcal{E}(t)}}\cr
&\leq Ce^{-\frac{a_{0}}{8}|p|^2},
\end{align*}
for sufficiently small $\mathcal{E}(t)$. This completes the proof of the claim.
Now we turn to the proof of the proposition:\newline

Using the claim above, and the H\"{o}lder inequality, we obtain
\begin{align*}
\bigg|\int_{\mathbb{R}^3}\partial_{\beta}^{\alpha}\Gamma(f)gdp\bigg| &=\sum_{1\leq i,j\leq5} \sum_{|\alpha_1|+|\alpha_2|+|\alpha_3|\leq|\alpha|}\int_{\mathbb{R}^3}\int_{0}^{1}\bigg|\partial_{\beta}^{\alpha_1}\mathcal{Q}_{i,j}^{\mathcal{F}}(\theta)\frac{\mathcal{F}(\theta)-\mathcal{F}(\theta)^2}{\sqrt{m-m^2}}(1-\theta)\bigg|d\theta	\cr
& \quad \times \langle\partial^{\alpha_2}f,e_i\rangle_{L^2_p}\langle\partial^{\alpha_3}f,e_j\rangle_{L^2_p}g dp\cr
&\leq C\sum_{|\alpha_1|+|\alpha_2|\leq|\alpha|}\int_{\mathbb{R}^3} e^{-\frac{a_{0}}{8}|p|^2} gdp||\partial^{\alpha_1}f||_{L_p^2}||\partial^{\alpha_2}f||_{L_p^2}\cr
&\leq C\sum_{|\alpha_1|+|\alpha_2|\leq|\alpha|}||e^{-\frac{a_{0}}{8}|p|^2}||_{L_p^2}|| g||_{L_p^2}||\partial^{\alpha_1}f||_{L_p^2}||\partial^{\alpha_2}f||_{L_p^2}\cr
&\leq C \sum_{|\alpha_1|+|\alpha_2|\leq|\alpha|}||\partial^{\alpha_1}f||_{L_p^2}||\partial^{\alpha_2}f||_{L_p^2}||g||_{L_p^2}.
\end{align*}
\end{proof}
%
%
%
%
%

\subsection{Local existence}
We now construct the local-in-time smooth solution:
\begin{theorem}\label{local}
Let $N\geq 4$ and $F_0=m+\sqrt{m-m^2}f_0\geq0$. Then there exists $M_0>0$, $T_*>0$, such that if $T_*\leq\frac{M_0}{2}$ and $\mathcal{E}(f_0)\leq\frac{M_0}{2}$, there exists a unique  non-negative local in time solution $f(x,p,t)$ of (\ref{f}) such that
\begin{enumerate}
\item The high order energy $\mathcal{E}(t)$ is uniformly bounded :
\begin{align*}
\sup_{0 \leq t\leq T_*}\mathcal{E}(t)\leq M_0.
\end{align*}
\item The high order energy $\mathcal{E}(t)$ is continuous in $[0,T_*)$.
\item The conservation laws (\ref{conservf}) hold for all $[0,T_*)$.
\item$(N,P,E)$ satisfies
\[
0<B(N,P,E)<\beta(-\ln3).
\]
\end{enumerate}
\end{theorem}
\begin{proof}
We define the iteration sequence $F^n$ as follow :
\begin{align}\label{iter}
\begin{split}
\partial_tF^{n+1}+p\cdot \nabla_xF^{n+1}&=\frac{1}{\tau(F^n)}(\mathcal{F}(F^n)-F^{n+1}),	\cr
F^{n+1}(x,p,0)&=F_0(x,p).
\end{split}
\end{align}
with $F^0(x,p,t)\equiv F_0(x,p)$. This is equivalent to
\begin{align}\label{fn}
\begin{split}
\partial_tf^{n+1}+p\cdot \nabla_xf^{n+1}+f^{n+1}&=Pf^n+\Gamma_1(f^n)+\Gamma_2(f^n,f^{n+1})+\Gamma_3(f^n),	\cr
f^{n+1}(x,p,0)&=f_0(x,p),
\end{split}
\end{align}
with $f^0(x,p,t)\equiv \frac{F_0-m}{\sqrt{m-m^2}}$.
Here, $\Gamma_1$ and $\Gamma_3$ are defined as in (\ref{Gamma}) whereas,  $\Gamma_2$ is defined slightly differently as
\[
\Gamma_2(f^n,f^{n+1})=\left\{\sum_{i=1}^5\mathcal{C}^{\tau}_i\langle f^n,e_i\rangle_{L^2_p}\right\}(Pf^n-f^{n+1}).
\]
The key ingredient is the uniform control of the size of high-order energy in each iteration step:
\begin{lemma}\label{5.2}
If $\mathcal{E}(f_0)<\frac{M_0}{2}$ then there exists $M_0>0$ and $T_*>0$ such that $\mathcal{E}(f^n(t))<M_0$
for all $n\geq0$ for $t\in [0,T_*]$.
\end{lemma}
\begin{proof}
We use induction. Assume we have obtained $f^n$ such that $\mathcal{E}(f^n(t))<M_0$
 on $[0,T_*]$, for sufficiently small $M_0$.
Then Lemma \ref{esN} (4) and Lemma \ref{esac} (1) imply that
\begin{align*}
0<B(N_n,P_n,E_n) < \beta(-\ln3)
\end{align*}
for sufficiently small $M_0$. Therefore, thanks to Theorem \ref{unique c}, we are able to find $a_n$ and $c_n$ such that $c_n>-\ln3$,  
which guarantees that $\mathcal{F}(F^n)$ is well-defined, and so is the iteration for $n+1$th step (\ref{iter}).
Then, applying the linearization argument in Section 3, we obtain (\ref{fn}).
Now, in view of Lemma \ref{esac} and Corollary \ref{beta lemma}, we see that $|\beta^{\prime}(c)|$ has a strictly positive lower bound, enabling one to compute derivatives of all the equilibrium coefficients by Lemma \ref{esdc} and \ref{esda}, and therefore, of $\mathcal{F}(F^n)$.
This implies that  $f^{n+1}$ also is smooth.
Hence, we can apply $\partial_{\beta}^{\alpha}$ on both sides of (\ref{fn}):
\begin{align*}
\partial_t\partial_{\beta}^{\alpha}f^{n+1}+p\cdot\nabla_x\partial_{\beta}^{\alpha}f^{n+1}+\partial_{\beta}^{\alpha}f^{n+1}&=-\sum_{i=1}^3 \partial^{\alpha+e_i}\partial_{\beta-e_i}f^{n+1} + \partial_{\beta}^{\alpha}Pf^n +\partial_{\beta}^{\alpha}\Gamma_1(f^n) \cr
&+ \partial_{\beta}^{\alpha}\Gamma_2(f^n,f^{n+1})+\partial_{\beta}^{\alpha}\Gamma_3(f^n),
\end{align*}
and take inner product with $\partial_{\beta}^{\alpha}f^{n+1}$. Then, a standard argument leads to
\begin{align*}
\big(1-CT_*-CT_*\sqrt{M_0}-CT_*M_0\,\big)\sup_{0 \leq t \leq T_*}\mathcal{E}(f^{n+1}(t)) \leq \Big(\,\frac{1}{2}+CT_*+CT_*\sqrt{M_0}+CT_*M_0\,\Big)M_0,
\end{align*}
which, for sufficiently small $M_0$, gives the desired result for $\mathcal{E}(f^{n+1})$.
This completes the proof of the lemma.
\end{proof}
With Lemma \ref{5.2}, the remaining proof for Theorem \ref{local} is standard. We omit it.
\end{proof}
%
%
%
%
%
%
%
%
%
\section{Proof of the main theorem}
Now, we have obtained most of the necessary estimates, and the remaining process is rather standard. We only sketch the proof.
\subsection{Coercivity of L}
We define $\bar{a}$, $\bar{b}$ and $\bar{c}$ as follows:
\begin{align*}
&\bar{a}(x,t)=\int_{\mathbb{R}^3}f\sqrt{m-m^2}dp,	\cr
&\bar{b}_i(x,t)=\int_{\mathbb{R}^3}fp_i\sqrt{m-m^2}dp, \quad  (i=1,2,3)	\cr &\bar{c}(x,t)=\int_{\mathbb{R}^3}f|p|^2\sqrt{m-m^2}dp,
\end{align*}
and
\begin{align*}
\tilde{P}f=\bar{a}(x,t)\sqrt{m-m^2}+\sum_{i}\bar{b}_i(x,t)p_i\sqrt{m-m^2}+\bar{c}(x,t)|p|^2\sqrt{m-m^2}.
\end{align*}
We then split $f$ into the macroscopic part $\tilde{P}f$ and the microscopic part $(I-\tilde{P})f$. Substituting this  in (\ref{f}), one gets
\begin{align*}
(\partial_t+p\cdot \nabla_x)(\tilde{P}f) = -(\partial_t+p\cdot \nabla_x)((I-\tilde{P})f)+L(I-\tilde{P})f+\Gamma(f).
\end{align*}
We then expand the l.h.s with respect to
\begin{align}\label{newbasis}
\{\sqrt{m-m^2},p_i\sqrt{m-m^2},p_ip_j\sqrt{m-m^2},p_i^2\sqrt{m-m^2},p_i|p|^2\sqrt{m-m^2}\},
\end{align}
for $i,j=1,2,3$ to write it as
\begin{align*}
\sum_{1\leq i\leq 3}\left\{\partial_{x_i}\bar{c}|p|^2+(\partial_{x_i}\bar{b}_i+\partial_t\bar{c})p_i^2+\sum_{ i<j\leq 3}(\partial_{x_i}\bar{b}_j+\partial_{x_j}\bar{b}_i)p_ip_j+(\partial_t\bar{b}_i+\partial_{x_i}\bar{a})+\partial_t\bar{a}\right\}\sqrt{m-m^2}.
\end{align*}
To arrive at the set of micro-macro equations:
\begin{lemma} $\bar{a}$, $\bar{b}$ and $\bar{c}$ satisfies the following system:
\begin{align}\label{sytem}
\begin{split}
\partial_t\bar{a} &= l_{\bar{a}}+h_{\bar{a}}, \cr
\partial_t\bar{b}_i + \partial_{x_i}\bar{a} &= l_{\bar{a}\bar{b}i}+h_{\bar{a}\bar{b}i},	\cr
\partial_{x_i}\bar{b}_j + \partial_{x_j}\bar{b}_i &= l_{ij}+h_{ij} \quad (i\neq j),	\cr
\partial_{x_i}\bar{b}_i+\partial_{t}\bar{c} &= l_{\bar{b}\bar{c}i} + h_{\bar{b}\bar{c}i},	\cr
\partial_{x_i}\bar{c} &= l_{\bar{c}i} + h_{\bar{c}i},
\end{split}
\end{align}
where $l_{\bar{a}},l_{\bar{a}\bar{b}i},l_{ij},l_{\bar{b}\bar{c}i},l_{\bar{c}i}$ are coefficient of the expansion of $-(\partial_t+p\cdot \nabla_x)((I-\tilde{P})f)+L(I-\tilde{P})f$ with respect to basis (\ref{newbasis}).
Similarly, $h_{\bar{a}},h_{\bar{a}\bar{b}i},h_{ij},h_{\bar{b}\bar{c}i},h_{\bar{c}i}$ are the coefficients of the expansion of $\Gamma(f)$ with respect to same basis.
\end{lemma}
From (\ref{system}), it is now standard to derive the following full coercivity estimate for sufficiently small $\mathcal{E}(t)$:
\begin{align}\label{coer}
\begin{split}
\sum_{|\alpha|\leq N} \langle L\partial^{\alpha}f,\partial^{\alpha}f\rangle_{L^2_{x,p}} \leq
- \delta \sum_{|\alpha|\leq N} ||\partial^{\alpha}f(t)||_{L_{x,p}^2}^2 .
\end{split}
\end{align}

\subsection{Global existence}
Finally, we extend the local existence to the global one by closing the nonlinear energy estimate. Let $f$ be the smooth local in time solution constructed in Theorem \ref{local}. First we derive the energy estimate with $|\beta|=0$.  Applying $\partial^{\alpha}$ on each side of (\ref{f})
and taking inner product with $\partial^{\alpha}f$, we obtain
\begin{align*}
\frac{1}{2}\frac{d}{dt} ||\partial^{\alpha}f||_{L^2_{x,p}}^2 = \langle L\partial^{\alpha}f,\partial^{\alpha}f\rangle_{L^2_{x,p}} + \langle\partial^{\alpha}\Gamma(f),\partial^{\alpha}f\rangle_{L^2_{x,p}}.
\end{align*}
By using coercive estimates (\ref{coer})  and nonlinear estimates in Proposition \ref{prop}, we derive
\begin{align*}
\mathcal{E}_0^{\alpha} : \quad \frac{1}{2}\frac{d}{dt} ||\partial^{\alpha}f||_{L^2_{x,p}}^2 + \delta \sum_{|\alpha|\leq N}||\partial^{\alpha}f||_{L^2_{x,p}}^2 \leq C\sqrt{\mathcal{E}(t)}\mathcal{E}(t),
\end{align*}
For the energy estimate involving velocity derivatives, we apply $\partial_{\beta}^{\alpha}$ to (\ref{f}):
\begin{align*}
\left\{\partial_t +p \cdot \nabla_x +1\right\}\partial^{\alpha}_{\beta}f = \partial_{\beta_1}p \cdot \nabla_x \partial_{\beta-\beta_1}^{\alpha}f + \partial_{\beta}P\partial^{\alpha}f + \partial_{\beta}^{\alpha}\Gamma (f),
\end{align*}
and take inner product with $\partial_{\beta}^{\alpha}f$ to get
\begin{align*}
\mathcal{E}_{\beta}^{\alpha} : \quad \frac{1}{2}\frac{d}{dt} ||\partial_{\beta}^{\alpha}f||_{L^2_{x,p}}^2 + ||\partial_{\beta}^{\alpha}f||_{L^2_{x,p}}^2
&\leq   \frac{1}{2\epsilon}\sum_{i=1}^3||\partial_{\beta-e_i}^{\alpha+e_i}f||_{L^2_{x,p}}^2+ \frac{3\epsilon}{2}||\partial_{\beta}^{\alpha}f||_{L^2_{x,p}}^2 \cr
& \quad +\frac{C}{2\epsilon}||\partial^{\alpha}f||_{L^2_{x,p}}^2+\frac{C\epsilon}{2}||\partial_{\beta}^{\alpha}f||_{L^2_{x,p}}^2+C\sqrt{\mathcal{E}(t)}\mathcal{E}(t).
\end{align*}
For sufficiently small $\epsilon$,  $||\partial_{\beta}^{\alpha}f||_{L^2_{x,p}}^2$ terms in right hand side are absorbed in the good term of the l.h.s to yield
\begin{align*}
\mathcal{E}_{\beta}^{\alpha} : \quad \frac{1}{2}\frac{d}{dt} ||\partial_{\beta}^{\alpha}f||_{L^2_{x,p}}^2 + \frac{1}{2}||\partial_{\beta}^{\alpha}f||_{L^2_{x,p}}^2
&\leq   C_{\epsilon}\sum_{i=1}^3||\partial_{\beta-e_i}^{\alpha+e_i}f||_{L^2_{x,p}}^2 +C_{\epsilon}||\partial^{\alpha}f||_{L^2_{x,p}}^2\cr
&\quad +C\sqrt{\mathcal{E}(t)}\mathcal{E}(t).
\end{align*}
Then, we observe that right hand side of $\displaystyle\sum_{|\beta|=m+1}\mathcal{E}_{\beta}^{\alpha}$ can be controlled by the good terms of
\begin{align*}
C_m\sum_{|\beta|=m}\mathcal{E}^{\alpha}_{\beta} + C_m \sum_{|\alpha|\leq N}\mathcal{E}_0^{\alpha},
\end{align*}
for sufficiently large $C_m$. Therefore, by standard induction argument, we obtain
\begin{align*}
\sum_{\substack{|\alpha|+|\beta|\leq N \cr |\beta|\leq m}} \left\{C_m\frac{d}{dt}||\partial_{\beta}^{\alpha}f||_{L^2_{x,p}}^2+ \delta_m||\partial_{\beta}^{\alpha}f||_{L^2_{x,p}}^2 \right\} \leq C_N \sqrt{\mathcal{E}(t)}\mathcal{E}(t),
\end{align*}
for constant $C_m$ and $\delta_m$. Then by standard continuity argument we derive global in time existence for smooth solution for (\ref{f}). This completes the proof.\newline

{\bf Acknowledgement}
The work of S.-B. Yun was supported by Samsung Science and Technology Foundation under
Project Number SSTF-BA1801-02.

\bibliographystyle{amsplain}

\end{document}